\documentclass[10pt]{amsart}
\usepackage{amsxtra, amsfonts, amsmath, amsthm, amstext, amssymb, amscd, mathrsfs, verbatim, color}
\usepackage{threeparttable}
\usepackage{mathtools}
\usepackage[ansinew]{inputenc}\usepackage[T1]{fontenc}
\usepackage[all,cmtip]{xy}
\theoremstyle{plain}

\newtheorem{theorem}{Theorem}[section]
\newtheorem{lemma}[theorem]{Lemma}
\newtheorem{definition-theorem}[theorem]{Definition-Theorem}
\newtheorem{proposition}[theorem]{Proposition}
\newtheorem{corollary}[theorem]{Corollary}

\theoremstyle{definition}
\newtheorem{definition}[theorem]{Definition}
\newtheorem{example}[theorem]{Example}
\newtheorem{remark}[theorem]{Remark}
\newtheorem{notation}[theorem]{Notation}
\newcommand \bth[1] { \begin{theorem}\label{t#1} }
\newcommand \ble[1] { \begin{lemma}\label{l#1} }

\newcommand \bpr[1] { \begin{proposition}\label{p#1} }
\newcommand \bco[1] { \begin{corollary}\label{c#1} }
\newcommand \bde[1] { \begin{definition}\label{d#1}\rm }
\newcommand \bex[1] { \begin{example}\label{e#1}\rm }
\newcommand \bre[1] { \begin{remark}\label{r#1}\rm }

\newcommand \bnota[1] {\begin{notation}\label{n#1}\rm }
\newcommand {\ele} { \end{lemma} }

\newcommand {\epr} { \end{proposition} }
\newcommand {\eco} { \end{corollary} }
\newcommand {\ede} { \end{definition} }
\newcommand {\eex} { \end{example} }
\newcommand {\ere} { \end{remark} }
\newcommand {\enota} { \end{notation} }

\def \calA  {{\mathcal{A}}}           

\def \calH  {{\mathcal{H}}}

\def \calS {{\mathcal{S}}}

\def \St { { \mathrm{St}}}

\DeclareMathOperator \tr { {\mathrm{tr}} }

\DeclareMathOperator \End { {\mathrm{End}} }

\DeclareMathOperator \Hom { {\mathrm{Hom}} }

\DeclareMathOperator \ind{ {\mathrm{ind}}}

\DeclareMathOperator \sgn { { \mathrm{sgn}}}

\DeclareMathOperator \vol { {\mathrm{vol}} }

\renewcommand \max { {\mathrm{max}} }
									
\begin{document}
\setlength{\baselineskip}{1.2\baselineskip}

\title[Bernstein-Zelevinsky derivatives]
{Bernstein-Zelevinsky derivatives: a Hecke algebra approach}

\author[Kei Yuen Chan]{Kei Yuen Chan}
\address{ Korteweg-de Vries Institute for Mathematics, Universiteit van Amsterdam}
\email{keiyuen.chan@gmail.com, K.Y.Chan@uva.nl}

\author[Gordan Savin]{Gordan Savin}
\address{
Department of Mathematics \\
University of Utah}
\email{savin@math.utah.edu}

\begin{abstract} 
Let $G$ be a general linear group over a $p$-adic field. 
It is well known that Bernstein components of the category of smooth representations of $G$ are described by Hecke algebras arising from Bushnell-Kutzko types. 
We describe the Bernstein components of the Gelfand-Graev representation of $G$ by explicit Hecke algebra modules. This result is used to translate the theory of 
Bernstein-Zelevinsky derivatives in the language of representations of Hecke algebras, where  we develop a comprehensive theory. 
\end{abstract}
\maketitle 

\section{Introduction}

Bernstein-Zelevinsky derivatives were first introduced and studied in \cite{BZ} and \cite{Ze} and are an important tool in the representation theory of general linear groups over 
$p$-adic fields. One goal of this paper is to formulate functors for Hecke algebras that correspond to Bernstein-Zelevinsky derivatives and show that Bernstein-Zelevinsky derivatives can be determined from the corresponding Hecke algebra functors.  
An advantage of our approach is that the some representations, such as generalized Speh modules,  have explicit description in terms of the corresponding Hecke algebra modules, rather than just being defined as Langlands quotients. Thus, 
as an application of our study, we compute the Bernstein-Zelevinsky derivatives of generalized Speh modules, by a method which does not use the determinantal formula of Tadi\'c \cite{Ta} and Lapid-M\'inguez \cite{LM} or Kazhdan-Lusztig polynomials \cite{Ze2, CG}.

\subsection{Main results}  
Let $F$ be a $p$-adic field. 
Let $G$ be a general linear group over $F$. The category  $\mathfrak{R}(G)$ of smooth representations of $G$ can be described by Hecke algebras arising from 
Bushnell-Kutzko types \cite{BK}. In order to keep notation simple, we shall only discuss the simple types. 
This restriction will present no loss of generality, as far as the theory of Bernstein-Zelevinsky derivatives is concerned. 
So let $G$ (or $G_n$ if we need to distinguish between the
general linear groups of different rank)  be the group $GL_{nr}(F)$ where $r$ is a fixed integer. The group $G$ contains a Levi group $L=GL_r(F)^n$. Let $\delta$ be 
a supercuspidal representation of $GL_r(F)$. Then $\tau=\delta\boxtimes \ldots \boxtimes \delta$ is a supercuspidal representation of $L$. The pair 
$\mathfrak s=[L, \tau]$ (or $\mathfrak s_n$) determines a Bernstein component $\mathfrak{R}^{\mathfrak s}(G)$ of $\mathfrak{R}(G)$. 

A type is a representation $\rho$ of an open compact subgroup $K$ of $G$. If $\pi$ is a smooth representation of $G$, then $\pi_{\rho}=(\pi\otimes \rho^{\vee})^K$ is naturally a 
module for $\calH(G,\rho)$, the Hecke algebra of $\End(\rho^{\vee})$-valued functions on $G$. A type $\rho$ is a Bushnell-Kutzko type if $\pi \mapsto \pi_{\rho}$ is an equivalence 
of $\mathfrak{R}^{\mathfrak s}(G)$ and the category of $\calH(G,\rho)$-modules. For $\mathfrak s_n$, described above, such type $\rho_n$ is constructed in \cite{BK} and  
 in \cite{Wa} in the tame case. 
Moreover, it is proved that $\calH(G,\rho_n)$ is isomorphic to $\calH_n$, the Iwahori Hecke algebra of $GL_n(F')$,
 where $F'$ is an extension of $F$ depending on $\rho_n$.
The Weyl group of $GL_n(F')$ is isomorphic to the group of permutation matrices $S_n$, 
and  $\calH_n$ has a finite-dimensional subalgebra $\calH_{S_n}$ with a basis $T_w$ of characteristic functions of double cosets of $w\in S_n$.  
The algebra $\mathcal H_{S_n}$ has a one dimensional representation $\sgn$, $T_w\mapsto (-1)^{l(w)}$,  where $l$ is the length function on $S_n$.

 Let $U$ be the unipotent subgroup of all strictly upper triangular matrices in $G$.  Let $\psi$ be a Whittaker character of $U$. 
 One of the main results of this paper is a description of the Bernstein components of the Gelfand-Graev representation $\mathrm{ind}_U^G \psi$ in terms of the Hecke algebra action:

\begin{theorem} (Theorem  \ref{thm GG}) \label{thm intro realize}
The $\mathcal H_n$-module $(\mathrm{ind}_{U}^{G}\psi)_{\rho_n}$ is isomorphic to $\mathcal H_n \otimes_{\mathcal H_{S_n}} \mathrm{sgn}$.
\end{theorem}

This theorem is proved in \cite{CS} for the component consisting of representations generated by their Iwahori-fixed vectors, by an explicit computation. 
 Here we give a more abstract proof using projectivity of $\mathrm{ind}_{U}^{G}\psi$ and that the Bernstein components of 
$\mathrm{ind}_U^G \psi$ are finitely generated, a result of Bushnell and Henniart \cite{BH}. Our result is therefore a refinement of theirs for the general linear group. 
Projectivity of $\mathrm{ind}_{U}^{G}\psi$ was proved by Prasad in \cite{Pr} by an argument very specific to general linear groups. In the appendix we prove projectivity 
of the Gelfand-Graev representation  in a very general setting.

Theorem \ref{thm intro realize} plays an important role in the formulation of the Bernstein-Zelevinsky derivatives in the language of Hecke algebras.  
  To that end, let 
\[ 
   \mathbf S_n=  (\sum_{w\in S_n} (1/q)^{l(w)})^{-1}  \sum_{w\in S_n} (-1/q)^{l(w)} T_w \in \mathcal H_n.  
   \] 
    If $\sigma$ is an $\mathcal H_{n}$-module, then $\mathbf S_n(\sigma)$ is the $\sgn$-isotypic subspace of $\sigma$.   
For every $i=1, \ldots, n-1$ we have  an embedding of the Hecke algebra  $\mathcal H_{n-i} \otimes \mathcal H_i$ into $ \mathcal H_n$.  In particular, 
the map $h\mapsto h\otimes 1$ realizes $\mathcal H_{n-i}$ as a subalgebra of $\mathcal H_{n}$. Let $\mathbf S_i^n$ be the image in 
$\mathcal H_n$ of $1\otimes \mathbf S_i$, where $\mathbf S_i$ is the sign projector in $\mathcal H_{i}$.  For every $\mathcal H_n$-module $\sigma$, 

\begin{align} \label{eqn intro ZD Hecke} 
\mathbf{BZ}_i(\sigma):=\mathbf S_i^n (\sigma) 
\end{align}
is naturally an $\mathcal H_{n-i}$-module. This is the $i$-th derivative of $\sigma$. Let $\pi$ be a smooth representation of $G_n$, and $\pi^{(l)}$ its $l$-th Bernstein-Zelevinsky
derivative. If $\pi$ is in  $\mathfrak{R}^{\mathfrak s_n}(G_n)$, then $\pi^{(l)}=0$ unless $l$ is a multiple of $r$, and then $\pi^{(ir)}$ is an object  in 
$\mathfrak{R}^{\mathfrak s_{n-i}}(G_{n-i})$.

\begin{theorem} (Theorem \ref{thm bz aha}) \label{thm intro bz aha}
Let $\pi$ be an admissible representation of $G_n$ in  $\mathfrak{R}^{\mathfrak s_n}(G_n)$.  
Let $\mathbf{BZ}_i$ be the functor defined in (\ref{eqn intro ZD Hecke}). There is a functorial isomorphism of $\mathcal H_{n-i}$-modules
\[    
 (\pi^{(ir)})_{\rho_{_{n-i}}} \cong\mathbf{BZ}_i(\pi_{\rho_n}). 
\]
\end{theorem}

One can similarly formulate Bernstein-Zelevinsky derivatives for graded Hecke algebras. We check in Sections \ref{s bz gha}  and  \ref{s bz ghaII} that Bernstein-Zelevinsky derivatives of affine Hecke algebras and graded Hecke algebras agree under the Lusztig's reductions. A reason for formulating the Bernstein-Zelevinsky derivatives for graded Hecke algebras is that one can apply representation theory of symmetric groups, in particular the Littlewood-Richardson rule, to compute the Bernstein-Zelevinsky derivatives of 
 generalized Speh representations, see Section \ref{s bz speh}, for details.

\subsection{Acknowledgements} 
This work was initiated during the Sphericity 2016 Conference in Germany. The authors would like to thank the organizers for providing the excellent environment for discussions. 
The authors would like to thank a referee for an informative report. 
The first author was supported by the Croucher Postdoctoral Fellowship. The second author was supported in part by NSF grant DMS-1359774.

\section{Affine Hecke algebra} \label{S:hecke} 

\subsection{}
Let $\mathcal H_n$ be the Iwahori-Hecke algebra of $GL(n)$ over the $p$-adic field $F'$. As an abstract algebra, $\mathcal H_n$ is generated by elements 
$T_1, \ldots ,T_{n-1}$  and by the algebra 
$\mathcal A_n= \mathbb C[x_1^{\pm 1}, \ldots, x_n^{\pm1}]$ of Laurent polynomials. The algebra $\calA_n$ is isomorphic to the group algebra of the lattice $\mathbb Z^n$, as follows.  
The group algebra is spanned by the elements $\theta_x$ where $x\in \mathbb Z^n$ with the multiplication $\theta_x \cdot \theta_y= \theta_{x+y}$. We can identify the two algebras by 
$\theta_x=x_1^{m_1} \cdots x_n^{m_n}$ where $x=(m_1, \ldots , m_n)\in \mathbb Z^n$.   We shall use both notations for elements in $\calA_n$ at our convenience. 
The elements $T_j$ satisfy the quadratic relation 
$(T_j+1)(T_j-q)=0$ (and braid relations) and the relationship between $T_j$ and $f\in \mathcal A_n$ is given by 
\begin{equation}\label{E:hecke} 
T_j f - f^{s_j} T_j = (q-1) x_j\frac{ f-f^{s_j}}{x_j-x_{j+1}} 
\end{equation} 
where $s_j$ is the permutation $(j,j+1)$ and $f^{s_j}$ is obtained from $f$ by permuting $x_j$ and $x_{j+1}$. 
The Weyl group of $GL(n)$ is isomorphic to the group of permutations $S_n$, and the center $\mathcal Z_n$ of $\mathcal H_n$ is equal to 
the subalgebra of $S_n$-invariant Laurent polynomials in $\mathcal A_n$.   We shall use the fact that $\mathcal A_n$ is a free $\mathcal Z_n$-module of rank $|S_n|$. Let 
$\mathcal H_{S_n}$ be the subalgebra of $\mathcal H_n$ generated by the elements $T_j$, $j=1,\ldots, n-1$. It is a finite algebra  spanned by elements $T_w$, $w\in S_n$, 
where $T_w$ is a product of $T_j$ as given by a shortest expression of $w$ as a product of simple reflections. In particular, the dimension of 
$\mathcal H_{S_n}$ is $|S_n|$. We shall also use the fact that 
the multiplication in $\calH_n$ of elements in $\calA_n$ and $\mathcal H_{S_n}$ gives isomorphisms 
\[ 
\calH_n \cong \calA_n\otimes_{\mathbb C}  \mathcal H_{S_n}  \cong \mathcal H_{S_n}\otimes_{\mathbb C}  \calA_n. 
\] 
The algebra $\mathcal H_{S_n}$ has two one-dimensional  
representations: the trivial,  where $T_j=q$ for all $j$,  and the sign representation,  where $T_j=-1$ for all $j$. A twisted Steinberg representation is a one-dimensional 
representation of $\mathcal H_n$ such that its restriction to $\mathcal H_{S_n}$  is the sign representation. 
This section is devoted to the proof of the following theorem. 

\begin{theorem} \label{T:hecke} 
Let $\Pi$ be an $\mathcal H_n$-module such that: 
\begin{itemize} 
\item $\Pi$ is projective and finitely generated. 
\item $\dim\mathrm{Hom}_{\mathcal H_n} (\Pi, \pi) \leq 1$ for an irreducible principal series representation $\pi$. 
\item A twisted Steinberg representation is a quotient of $\Pi$. 
\end{itemize} 
Then $\Pi$ is isomorphic to $\mathcal H_n\otimes_{\mathcal H_{S_n}}\mathrm{sgn}$.
\end{theorem}

\begin{lemma} \label{L:free}
Let $P$ be a projective and finitely-generated ${\mathcal H_n}$-module. Then $P$ is free and finitely generated as an ${\mathcal A_n}$-module.
\end{lemma}
\begin{proof}
Let $\sigma$ be an $\mathcal A_n$-module. Recall that we have a natural isomorphism
\[  
\mathrm{Hom}_{\mathcal A_n}(P|_{\mathcal A_n}, \sigma) \cong \mathrm{Hom}_{\mathcal H_n}(P, \mathrm{Hom}_{\mathcal A_n}({\mathcal H_n}, \sigma)) .
\]
Since $\mathcal H_n$ is a free $\mathcal A_n$-module, $ \mathrm{Hom}_{\mathcal A_n}(\mathcal H_n, \sigma)$ is exact in $\sigma$. 
Since $P$ is a projective $\mathcal H_n$-module, the above 
isomorphism implies that  $P$, viewed as an $\mathcal A_n$-module, is also projective.  It is clear that $P$ is finitely-generated $\mathcal A_n$-module, hence 
 $P$ is a free $\mathcal A_n$-module by a version of the Quillen-Suslin theorem for rings of Laurent polynomials due to Swan \cite{Sw}. 
\end{proof}

Lemma, combined with the first assumption on $\Pi$, implies that $\Pi\cong \mathcal A_n^r$ as an $\mathcal A_n$-module. 

\begin{lemma} \label{L:extension}
Let $\pi$ be an irreducible principal series module  of $\mathcal H_n$ i.e. an irreducible representation whose dimension is equal to the order of $S_n$. 
Let $\mathcal J$ be the annihilator of $\pi$ in the center $\mathcal Z_n$ of $\mathcal H_n$. 
Let  
\[  0 \rightarrow \pi'' \rightarrow \pi' \stackrel{g}{\rightarrow} \pi \rightarrow 0 
\]
be a non-split exact sequence $\mathcal H_n$-modules. Then $\mathcal J \pi' \neq 0$. 

\end{lemma}

\begin{proof} We abbreviate $P_{\mathrm{sgn}}= \mathcal H_n \otimes_{\mathcal H_{S_n}}\mathrm{sgn}$.
 By the projectivity of $P_{\mathrm{sgn}}$  we have the following maps 
\[   P_{\mathrm{sgn}} \stackrel{f}{\rightarrow} \pi' \stackrel{g} {\rightarrow}\pi \rightarrow 0
\]
such that the composition $g\circ f$ is non-trivial.  Since $\mathcal J \pi=0$, the composition $g\circ f$ descends to a map from 
$P_{\mathrm{sgn}} /\mathcal JP_{\mathrm{sgn}}$ to $\pi$. Since 
\[\dim_{\mathbb C} ( P_{\mathrm{sgn}} /\mathcal JP_{\mathrm{sgn}}) = \dim_{\mathbb C}(\mathcal A_n\mathcal /\mathcal J\mathcal A_n)= |S_n|=\dim_{\mathbb C}(\pi)\] 
the composition $g\circ f$ descends to an isomorphism  $P_{\mathrm{sgn}} /\mathcal JP_{\mathrm{sgn}} \cong \pi$. If  $\mathcal J \pi' = 0$ then $f$ descends to 
a map from $P_{\mathrm{sgn}} /\mathcal JP_{\mathrm{sgn}}\cong \pi$ to $\pi'$, contradicting the assumption on the exact sequence. 
 \end{proof}

Let $\mathcal J$ and $\pi$ be as in the lemma. Then $\Pi/\mathcal J\Pi$ has a composition series such that any irreducible subquotient is isomorphic to $\pi$. Since 
$\Pi/\mathcal J\Pi$ is annihilated by $\mathcal J$, by an easy application of  Lemma \ref{L:extension},  it is 
 a direct sum of $r$ copies of $\pi$. Hence $r=1$, by the second assumption on $\Pi$. 

\begin{lemma} \label{L:structure}
Let $P$ be an $\mathcal H_n$-module isomorphic to $\mathcal A_n$, as an $\mathcal A_n$-module. Then 
$P$ is isomorphic to $\mathcal H_n \otimes_{\mathcal H_{S_n}} \mathrm{sgn}$  or  $ \mathcal H_n \otimes_{\mathcal H_{S_n}} \mathrm{1}$. 
\end{lemma} 

A proof of this lemma is in the next section. 
 Lemma, combined with the third assumption on $\Pi$, implies that $\Pi$ is isomorphic to $\mathcal H_n \otimes_{\mathcal H_{S_n}} \mathrm{sgn}$. 
This completes the proof of Theorem \ref{T:hecke}. 

\begin{remark}
The authors would like to thank a referee for pointing out that the structure of finitely generated projective modules can be 
understood from $K$-theory of affine Hecke algebras \cite[Section 5.1]{So}. 
Some explicit $K$-theoretic computations can be found in \cite[Chapter 6]{So0}. 
\end{remark}

\subsection{$\mathcal H_n$-module structure on $\mathcal A_n$}  The main goal of this section is to prove Lemma \ref{L:structure}. This will be accomplished 
by an explicit calculation for $\calH_2$, from which we shall derive the general case. 
We work in a more general setting, and replace $\mathcal A_2$ with $A[x_1^{\pm1}, x_2^{\pm 1}]$ where $A$ is a $\mathbb C$-algebra. 
So assume we have an $\mathcal H_2$-structure on $A[x_1^{\pm1}, x_2^{\pm 1}]$. In particular, if $g(x_1,x_2) \in A[x_1^{\pm1}, x_2^{\pm 1}]$ is invertible, then 
\[ 
T_1( g(x_1,x_2)) = f(x_1, x_2 )g(x_1,x_2) 
\] 
for some $f(x_1,x_2) \in A[x_1^{\pm1}, x_2^{\pm 1}]$, depending on $g(x_1,x_2)$.  
Using the relation (\ref{E:hecke}), the relation $T_1^2 = (q-1)T_1 +q$ implies that $f(x_1,x_2)$ satisfies the following polynomial equation: 
\[ 
f(x_1,x_2) f(x_2,x_1) = (q-1)(  \frac{ x_1f(x_2,x_1) -x_2 f(x_1,x_2)}{x_1-x_2} )  +q. 
\] 
So our task is to solve this polynomial equation. To that end, we abbreviate 
\[ 
\tilde f(x_1,x_2) =   \frac{ x_1f(x_2,x_1) - x_2f(x_1,x_2)}{x_1-x_2}, 
\] 
and derive some explicit formulae for $\tilde f$. Assume that $f(x_1,x_2)= x_1^n x_2^m$.  If $m\geq n$ then 
\[ 
\tilde f(x_1,x_2) = x_1^mx_2 ^n + x_1^{m-1} x_2^{n+1} + \ldots +  x_1^n x_2^m. 
\] 
If $m<n$ then 
\[ 
\tilde f(x_1,x_2) = -x_1^{n-1}x_2 ^{m+1}  - \ldots - x_1^{m+1} x_2^{n-1}.   
\] 
Write $f(x_1,x_2)= \sum a_{n,m} x_1^n x_2^m$ and define 
\[  \mathrm{maxdeg}(f(x_1,x_2))=\max \left\{ n+m\in \mathbb{Z}: a_{n,m} \neq 0 \right\},
\]
\[  \mathrm{mindeg}(f(x_1, x_2)) = \min \left\{ n+m \in \mathbb{Z}: a_{n,m} \neq 0 \right\} .
\]
\begin{lemma} Assume that $f(x_1,x_2)\in A[x_1^{\pm1}, x_2^{\pm 1}]$ is a solution of the equation 
\[ 
f(x_2,x_1)f(x_1,x_2) = (q-1)\tilde f(x_1,x_2) + q. 
\] 
Then $\mathrm{maxdeg}(f(x_1,x_2))=\mathrm{mindeg}(f(x_1,x_2))=0$. 
\end{lemma} 
\begin{proof} Let $f(x_1,x_2)\in A[x_1^{\pm1}, x_2^{\pm 1}]$.  If $\mathrm{maxdeg}(f(x_1,x_2))=\mathrm{mindeg}(f(x_1,x_2))=0$ fails then 
$\mathrm{maxdeg}(f(x_1,x_2))>0$ or $\mathrm{mindeg}(f(x_1,x_2))<0$. Assume $\mathrm{maxdeg}(f(x_1,x_2))>0$. Then 
\[ 
\mathrm{maxdeg}(f(x_1,x_2)f(x_2,x_1)) > \mathrm{maxdeg}(f(x_1,x_2))\geq\mathrm{maxdeg}((q-1)\tilde f(x_1,x_2)+q))
\] 
so $f(x_1,x_2)$ is not a solution. The case $\mathrm{mindeg}(f(x_1,x_2))<0$ is dealt with similarly. 
\end{proof} 

\vskip 5pt 
Lemma implies that a solution of the polynomial equation is a Laurent polynomial $f(x)$ where $x=x_2/x_1$.  We abbreviate 
\[ 
\tilde f(x) =  \frac{x^{-1/2} f(x^{-1}) - x^{1/2}f(x)}{x^{-1/2}-x^{1/2}}. 
\] 
\begin{lemma} Let $f(x)\in A[x^{\pm 1}]$ be a solution of 
\[ 
f(x)f(x^{-1})= (q-1)( \tilde f(x)) + q. 
\] 
Then there exists an integer $m$ and $\lambda =-1$ or $q$ such that $f(x) = f^{\lambda}_m(x)$ where, if $m\geq 0$, 
\[ 
f_m^{\lambda}= (q-1)(1 + x + \ldots  + x^{m-1}) + \lambda x^m
\] 
and, if $m<0$, 
\[ 
f_m^{\lambda}= -(q-1)(x^{-1} + \ldots + x^{m+1}) -\lambda x^m. 
\] 
\end{lemma} 
\begin{proof} 
It is trivial to check that a solution $f(x)$ cannot have at the same time negative and positive powers of $x$. So assume firstly that $f(x) =a_mx^m + a_{m-1}x^{m-1} + \ldots + a_0$, 
where $m\geq 0$ and $a_m\neq 0$. Since 
\[ 
\tilde f(x)= a_m x^m + (a_m + a_{m-1}) x^{m-1}  + \ldots + (a_m+ a_{m-1} + \ldots + a_0) + \ldots  + a_m x^{-m}, 
\] 
equating coefficients  of the two sides yields the following sequence of equations: 
\[ a_m a_0= (q-1)a_m \] 
\[ a_ma_1 + a_{m-1} a_0= (q-1) (a_m + a_{m-1})\] 
etc and the last 
\[ 
a_m^2 + a_{m-1}^2 + \ldots + a_0^2 = (q-1)(a_m+ a_{m-1} + \ldots + a_0) +q.
\] 
Since $a_m\neq 0$ the first equation implies $a_0=q-1$. Then the second implies that $a_1=q-1$ etc. 
Finally, the last implies that $a_m^2 =(q-1)a_m +q$, and this has two solutions, $-1$ and $q$. 
Now assume that $f(x)=-a_0 -a_{-1} x^{-1} - \ldots  -a_{m}x^{m}$, for $m<0$  and $a_m\neq 0$. Then  
\[ 
\tilde f(x)=a_{m} x^{m+1} + (a_m+ a_{m+1})x^{m+2} + \ldots + (a_m+ a_{m+1} + \ldots  + a_{-1}) + \ldots  + a_m x^{1-m}. 
\] 
In particular, we do not have the $x^m$ term. A comparison with the left hand side implies that $a_0=0$. The rest of the proof proceeds along the same lines as in the first case, 
giving $a_{-1}=(q-1)$, $a_{-2}=(q-1)$ ... and $a_m$ a solution of  $a_m^2 =(q-1)a_m +q$. 
\end{proof} 

\begin{corollary}  \label{C:action} 
Assume we have an $\mathcal H_2$-module structure on $A[x_1^{\pm 1}, x_2^{\pm1}]$. 
Then for every invertible $g(x_1,x_2)\in  A[x_1^{\pm 1}, x_1^{\pm1}]$  there exists an integer $m$ such that 
$g(x_1,x_2)x_2^{m}$ is an eigenvector of $T_1$. 
\end{corollary} 
\begin{proof} Two lemmas imply that $T_1( g(x_1,x_2))= f^{\lambda}_{m}(x_2/x_1) g(x_1,x_2)$. 
Now one checks that $g(x_1,x_2)x_2^{m}$ is an eigenvector. 
\end{proof} 

In view of the tensor product decomposition $\calH_n \cong \calA_n\otimes_{\mathbb C} \calH_{S_n}$, the following corollary completes the proof of Lemma \ref{L:structure}: 

\begin{corollary}   Assume we have an $\mathcal H_n$-module structure on $\mathcal A_n=\mathbb C[x_1^{\pm 1}, \ldots, x_n^{\pm1}]$. 
Then there exists an invertible element in $\mathcal A_n$ that is an eigenvector for $T_1, \ldots  ,T_{n-1}$. 
\end{corollary} 
\begin{proof} We apply Corollary \ref{C:action} to $A[x_1^{\pm 1}, x_2^{\pm1}]$ and $g(x_1,x_2)=1$, where $A=\mathbb C[x_i^{\pm 1}]$ for $i\neq 1,2$.  Thus there is an integer $m_2$ 
such that $x_2^{m_2}$ is an eigenvector of $T_1$. Next, we apply Corollary \ref{C:action} to  $A[x_2^{\pm 1}, x_3^{\pm1}]$ and $g(x_1,x_2)=x_2^{m_2}$, 
where $A=\mathbb C[x_i^{\pm 1}]$ for $i\neq 2,3$. 
Hence there exists an integer $m_3$ such that $x_2^{m_2} x_3^{m_3}$ is an eigenvector of $T_2$. Since $T_1$ and $x_3$ commute, $x_2^{m_2} x_3^{m_3}$ is still an eigenvector 
of $T_1$. Continuing in this fashion, we arrive to a monomial in $\mathbb C[x_1^{\pm 1}, \ldots, x_n^{\pm1}]$ that is a joint eigenvector for all $T_j$. 
\end{proof} 

\section{Gelfand-Graev representation} 

\subsection{Hecke algebras} Let $G$ be a $p$-adic reductive group. 
Let $K$ be an open compact subgroup of $G$ and $(\rho, E)$ a smooth, finite-dimensional, representation of $K$. Let 
$\mathcal H(G,\rho)$ be the algebra of compactly supported $\End(E^{\vee})$-valued functions on $G$ such that $f(kgk')= \rho^{\vee}(k) f(g) \rho^{\vee}(k')$ for $k,k'\in K$. 
Let $\calS(G)$ be the space of locally constant, compactly supported functions on $G$, and let $e_{\rho}\in \calS(G)$ be defined by 
\[ 
e_{\rho}(x)=\frac{\dim(\rho)}{\vol(K)} \tr_E(x^{-1}) 
\] 
if $x\in K$ and $0$ otherwise. Then $e_{\rho} \ast e_{\rho}= e_{\rho}$. Let 
$\calH_{\rho}=e_{\rho} \ast \calS(G) \ast e_{\rho}$. The two algebras are related by a canonical  isomorphism $\calH_{\rho} \cong \mathcal H(G,\rho)\otimes \End(E)$, see \cite{BK2}. 
 If $(\pi, V)$ is a smooth representation of $G$,  let 
\[ 
V_{\rho} =\Hom_K(E, V) \cong (E^{\vee} \otimes V)^K. 
\] 
Note that $f\in \mathcal H(G,\rho)$  naturally acts on $e^{\vee}\otimes v\in E^{\vee} \otimes V$ by the formula 
\[ 
\pi_{\rho}(f) (e^{\vee}\otimes v)= \int_G f(g)(e^{\vee}) \otimes \pi(g)(v) ~dg. 
\] 
This action preserves the subspace $(E^{\vee} \otimes V)^K$, and defines a structure of $\calH(G,\rho)$-module on $V_{\rho}$. 
On the other hand, 
\[ 
\pi(e_{\rho}) \cdot V\cong V_{\rho} \otimes E
\] 
 is naturally a $\calH_{\rho}$-module. The two structures are compatible with respect to the isomorphism $\calH_{\rho}\cong \mathcal H(G,\rho)\otimes \End(E)$.
 
 \begin{lemma} \label{L:finite} 
  Assume that a smooth representation $(\pi, V)$  of $G$ is finitely generated. Then $V_{\rho}$ is finitely generated $\calH(G, \rho)$-module. 
 \end{lemma} 
 \begin{proof} It suffices to show that $\pi(e_{\rho}) \cdot V$ is finitely generated $\calH_{\rho}$-module. Let $V_0$ be a finite-dimensional subspace generating $V$. Let 
 $J$ be an open compact subgroup of $K$ such that $V_0\subseteq V^J$. Then $V= \pi(\calS(G/J))\cdot V_0$. Assume, in addition, that $J$  is contained in 
 the kernel of $\rho$, so $e_{\rho}\in  \calS(J\backslash G /J)$. Then 
 \[ 
 \pi(e_{\rho} )\cdot V= \pi(e_{\rho}\ast \calS(J\backslash G /J))\cdot V_0. 
 \] 
 It is known that $\calS(J\backslash G /J)$ is finite over its center $Z_J$. Hence $\pi(e_{\rho}) \cdot V$ is finite over $e_{\rho} \ast Z_J\subseteq \calH_{\rho}$. 
 
 \end{proof} 
 
  Let $f\in \calH(G,\rho)$. Then $f(g)\in \End(E)$. Let $\bar{f}(g)$ be the image of $f(g)$ under the composite of the following isomorphisms: 
 \[
 \End(E)\cong E\otimes E^{\vee} \cong \End(E^{\vee}).
 \] 
 Let $f^*(g)=\bar f(g^{-1})$. Then the map $f\mapsto f^*$ is anti-isomorphism of $ \calH(G,\rho)$ and $\calH(G,\rho^{\vee})$. 
 Let $(\pi^{\vee}, V^{\vee})$ be the smooth dual of $(\pi, V)$. Then $V^{\vee}_{\rho^{\vee}}$ is an $\calH(G,\rho^{\vee})$-module. 
 We have a natural isomorphism 
 \[ 
  (V_{\rho})^*= ((E^{\vee} \otimes V)^K)^*\cong (E \otimes V^{\vee})^K= V^{\vee}_{\rho^{\vee}}
 \] 
  of  vector spaces where  $(V_{\rho})^*$ is the linear dual of $V_{\rho}$.  On $(V_{\rho})^*$ we have an anti-action $\pi^*_{\rho}$ of $\calH(G,\rho)$. Via the isomorphism 
  $(V_{\rho})^*\cong V^{\vee}_{\rho^{\vee}}$ the two actions are related by the formula 
  \[ 
  \pi^*_{\rho}(f)= \pi_{\rho^{\vee}}(f^*).
  \] 

 \subsection{Bernstein's decomposition} \label{ss types}

Let $\mathfrak{R}(G)$ be the category of smooth representations of $G$. We recall some notions and properties of Bernstein decomposition, and the Bushnell-Kutzko theory of types \cite{BK, BK2}, mainly for the case of general linear groups. 

Let $\mathfrak{B}(G)$ be the set of $G$-inertial equivalence classes. For each $\mathfrak{s} \in \mathfrak{B}(G)$, let $\mathfrak{R}^{\mathfrak{s}}(G)$ be the Bernstein component associated to $\mathfrak{s}$. More precisely, an inertial equivalence class $\mathfrak{s}$ consists of pairs $(L, \tau)$, where $L$ is a Levi subgroup of $G$ and $\tau$ is a supercuspidal representation, and $\mathfrak{R}^{\mathfrak{s}}(G)$ is the full subcategory of $\mathfrak{R}(G)$ whose objects have the property that every irreducible subquotient appears as a composition factor of $\mathrm{Ind}_P^{G}(\tau \otimes \chi)$ for some unramified character $\chi$ of  $L$ and  $P$ is a parabolic subgroup with the Levi part $L$.
Two pairs $(L_1, \tau_1)$ and $(L_2,\tau_2)$ are in the same equivalence class if and only if they determine the same subcategories in $\mathfrak{R}(G)$. 
The Bernstein decomposition asserts that there is an equivalence of categories:
\[  \mathfrak{R}(G) \cong \prod_{\mathfrak{s}\in \mathfrak{B}(G)} \mathfrak{R}^{\mathfrak{s}}(G) .
\]

\begin{definition} Fix an inertial equivalence class $\mathfrak s$.  Let $K$ be an open compact subgroup of $G$. 
 Let $\rho$ be a smooth finite-dimensional representation of $K$. Then $\rho$ is called  an $\mathfrak s$-type if $V \mapsto V_{\rho}$ is an equivalence of the category 
$ \mathfrak{R}^{\mathfrak{s}}(G)$  and the category of $\calH(G,\rho)$-modules. 
\end{definition} 

We now look at the special case when $G=GL_{nr}(F)$  and an inertial equivalence class $\mathfrak s_n$ is given by 
\begin{equation}  \label{eqn prod levi}
   L= GL_r(F) \times \ldots \times GL_r(F)
\end{equation}
and
\begin{equation} \label{eqn prod sc}
 \tau = \delta \boxtimes \ldots \boxtimes \delta,
\end{equation}
where $\delta$ is a supercuspidal representation of $GL_r(F)$ and the number of factors is $n$. Let $P$ the parabolic subgroup of $GL_{nr}(F)$, with the Levi $L$,  
consisting of block upper-triangular matrices. 
Let $\mathrm{St}_n(\delta)$ be the unique irreducible quotient of 
\[ \mathrm{Ind}_{P}^{G} (\nu^{\frac{1-n}{2}} \delta \boxtimes \nu^{\frac{3-n}{2}} \delta \boxtimes \ldots \boxtimes \nu^{\frac{n-1}{2}}\delta )
\]
as in \cite[Sec. 9.1]{BZ}. Then $\mathrm{St}_n(\delta)$ is an essentially square integrable representation,  also known as the generalized Steinberg representation. 
We have the following result due to Bushnell and Kutzko (and Waldspurger \cite{Wa} in the tame case):

\begin{theorem} \label{thm theory of type} Let $\mathfrak{s}_n$ be the inertial class of $G=GL_{nr}(F)$ as above. Then there exists an $\mathfrak{s}_n$-type $\rho_n$ and 
an isomorphism $\calH(G,\rho_n)\cong \calH_n$, where $\calH_n$ is defined in Section \ref{S:hecke} with $q$ equal to a power of the order of the residual field of $F$. Moreover, under the isomorphism $\calH(G,\rho_n)\cong \calH_n$, the generalized Steinberg representation $\St_n(\delta)_{\rho_n}$ corresponds to the Steinberg module of $\calH_n$. 
\end{theorem}

Let $U$ be the unipotent group of upper-triangular matrices in $G$. Let $\psi: U \rightarrow \mathbb C^{\times}$ be a Whittaker functional. 
The Gelfand-Graev representation is the induced representation $\ind_U^G(\psi)$, consisting of functions on $G$ with compact support modulo $U$. 

\begin{theorem} \label{thm GG}  Let $G=GL_{nr}(F)$ and let $\rho_n$ be the $\mathfrak{s}_n$-type as in Theorem \ref{thm theory of type}. Then 
\[ 
(\ind_U^G(\psi))_{\rho_n} \cong \mathcal H_n \otimes_{\calH_{S_n}} \mathrm{sgn} 
\] 
as $\calH(G,\rho_n) \cong \calH_n$-modules. 
\end{theorem} 
\begin{proof} We need to show that the conditions of Theorem \ref{T:hecke} are satisfied. By a result of Bushnell and Henniart \cite{BH}, every Bernstein component of the 
Gelfand-Graev representation is finitely generated. Thus $(\ind_U^G(\psi))_{\rho_n}$ is finitely generated $\calH(G,\rho_n)$-module by Lemma \ref{L:finite}. Moreover, 
the Gelfand-Graev representation is projective by Corollary \ref{C:projective}. Thus the first bullet in Theorem \ref{T:hecke} holds. The second bullet 
holds since any Whittaker generic representation appears as a quotient, with multiplicity one, of the Gelfand-Graev representation. Finally,  $\St_n(\delta)_{\rho_n}$  is an 
essentially discrete series representation and therefore Whittaker generic. Hence the third bullet holds. 
\end{proof} 

In addition to the isomorphism $\epsilon: \calH_n \rightarrow \calH(G,\rho_n)$, there is also an isomorphism $\epsilon^{\vee}: \calH_n \rightarrow \calH(G,\rho^{\vee}_n)$.  Since 
$(\epsilon(T_j))^*$ is supported on the same double coset as $\epsilon^{\vee}(T_j)$, and satisfies the same quadratic equation, the two elements must be the same. 
Hence the following diagram commutes, here the left vertical arrow is the anti-involution of $\calH_{S_n}$ defined by $T_j^*=T_j$ for all $j=1, \ldots, n-1$. 
\[ 
\xymatrix{
\calH_{S_n}   \ar[r]^{\epsilon} \ar[d]^{*} & \mathcal H(G, \rho_n) \ar[d]^{*} \\%
\calH_{S_n}	 \ar[r]^{\epsilon^{\vee}} & \mathcal H(G, \rho^{\vee}_n) 
},
\]

If $(\pi, V)$ is a smooth representation of $G$, let $V_{U,\psi}$ be the maximal quotient of $V$ such that $U$ acts on it by $\psi$. Recall that 
\[ 
\mathbf S_n= (\sum_{w\in S_n} (1/q)^{l(w)} )^{-1} \sum_{w\in S_n} (-1/q)^{l(w)} T_w, 
\] 
 where $l$ is the length function on $\mathbf S_n$, is the sign projector.

\begin{theorem} \label{thm functorial}  Let $G=GL_{nr}(F)$ and let $\rho_n$ be the $\mathfrak{s}$-type as in Theorem \ref{thm theory of type}. Let $(\pi,V)$ be 
an admissible  representation of $G$ in the component $\mathfrak{R}^{\mathfrak{s}}(G)$. Then there exists a functorial isomorphism of vector spaces 
$ \phi_V: S_n(V_{\rho_n}) \rightarrow V_{U,\bar\psi}$.   
\end{theorem} 
\begin{proof} We need the following: 

\begin{lemma} \label{lem GG} For every smooth representation $V$ in the component $\mathfrak{R}^{\mathfrak{s}}(G)$ 
 and every finite dimensional complex vector space $X$, there is an isomorphism,  functorial in $V$ and $X$, 
\[ 
\Phi_X:  \Hom_{\mathbb C}(V_{U,\bar\psi}, X) \rightarrow \Hom_{\mathbb C}(\mathbf S_n(V_{\rho_n}), X).  
\] 

\end{lemma} 
\begin{proof}  We start by observing some facts that will be needed in the proof. 
Let $Y$ and $Z$ be two complex vector spaces, and $Y^*$ and $Z^*$ their linear duals.  Then 
\[ 
\Hom_{\mathbb C}(Y, Z^*) \cong \Hom_{\mathbb C}(Z, Y^*).
\] 
If $Y$ and $Z$ are $\mathcal H_{S_n}$-modules, then $Y^*$ and $Z^*$ are $\mathcal H_{S_n}$-modules, where the action is the natural anti-action, precomposed 
with the anti-involution $T_j^*=T_j$ for all $j=1, \ldots, n-1$. Then 
\[ 
\Hom_{\calH_{S_n}}(Y, Z^*) \cong \Hom_{\calH_{S_n}}(Z, Y^*).
\] 
If $Y$ and $Z$ are smooth representations of $G$, let $Y^{\vee}$ and $Z^{\vee}$ be the smooth duals of $Y$ and $Z$. Then 
\[ 
\Hom_{G}(Y,  Z^{\vee}) \cong \Hom_{G}(Z, Y^{\vee}).
\] 
For every finite dimensional vector space $X$ we have the following sequence of isomorphisms: 
\begin{align*}
\Hom_{\mathbb C} (V_{U,\bar \psi},X)&\cong  \mathrm{Hom}_{G}(V, \mathrm{Ind}_{U}^{G} (X\boxtimes \bar \psi))  \quad \mbox{ (by Frobenius reciprocity) } \\
             & \cong \mathrm{Hom}_{G}(\mathrm{ind}_{U}^{G} (X^*\boxtimes \psi), V^{\vee}) \quad 
             \mbox{ (since $(\mathrm{ind}_{U}^{G} (X^*\boxtimes \psi)^{\vee}\cong \mathrm{Ind}_{U}^{G} (X\boxtimes \bar \psi)$)  } \\
                &  \cong \mathrm{Hom}_{\mathcal H_n}(\mathcal H_n \otimes_{\mathcal H_{S_n}} (X^*\boxtimes \mathrm{sgn}),  V^{\vee}_{\rho_n^{\vee}}) \quad 
                \mbox{ (by Theorem  \ref{thm GG} for $\rho_n^{\vee}$)}  \\
                &  \cong \mathrm{Hom}_{\mathcal H_{S_n}} (X^*\boxtimes \mathrm{sgn},  V^{\vee}_{\rho^{\vee}_n})  \quad \mbox{ (by Frobenius reciprocity) } \\
	&  \cong \mathrm{Hom}_{\mathcal H_{S_n}}(V_{\rho_n}, X\boxtimes \mathrm{sgn}) \mbox{ (since $V^{\vee}_{\rho^{\vee}_n}\cong (V_{\rho_n})^*$) } \\
								& \cong \Hom_{\mathbb C}(\mathbf S_n(V_{\rho_n}), X).\\
\end{align*}
The map $\Phi_X$ is the composite of the sequence of isomorphisms.  
\end{proof} 

Now assume that $V$ is admissible. Then $V_{U,\bar \psi}$ and $\mathbf S_n(V_{\rho_n})$ are finite-dimensional. By the Yoneda Lemma, Lemma \ref{lem GG} implies Theorem \ref{thm functorial}.

\end{proof}

Let $I$ be an Iwahori subgroup of $G$.  In the case when $(\pi,V)$ belongs to the 
Bernstein component of representations generated by their $I$-fixed vectors Theorem \ref{thm functorial} holds for all smooth representations, that is, 
without the admissibility assumption. This is Corollary 4.5 in \cite{CS} which is proved using an explicit version of Theorem \ref{thm GG} available in the Iwahori case. 
In this case $\mathbf{S}_n(V_{\rho_n})$ is simply $\mathbf{S}_n(V^{I})$. 
 The inclusion of  $\mathbf{S}_n(V^I)$ into $V$ followed with the projection on $V_{U,\bar\psi}$ gives the map $\phi_V$.

\section{Bernstein-Zelevinsky derivatives} 
In this section we shall change notation slightly and write $G_n=GL_{nr}(F)$.  
 We shall also use $\pi$ to denote the space of a smooth representation of $G_n$. 
As previously, $\rho_n$ is an $\mathfrak{s}_n$-type. 
 
\subsection{Jacquet functor} \label{s:jacquet} 
Let $P=MN$ be the minimal parabolic subgroup of $G_n$ of block-upper triangular matrices, with the Levi $M=G_{n-i} \times G_{i}$, and the unipotent radical $N$.
The restriction of the $K$-type $\rho_n$ to $K_M=K\cap M$ is irreducible and isomorphic to $\rho_{n-i}\boxtimes\rho_i$. 
 We have the following commutative diagram, a consequence of Theorem (7.6.20) in \cite{BK}:
\[ 
\xymatrix{
\calH_{n-i}\otimes \calH_i   \ar[r]^{\cong} \ar[d]^{m} & \mathcal H(M, \rho_{n-i}\boxtimes\rho_i) \ar[d]  \\%
\calH_n	 \ar[r]^{\cong} & \mathcal H(G, \rho_n) }
\]
where the vertical maps are injections. The left vertical map $m$ is explicitly described as follows: 
$m(T_j\otimes 1)  \mapsto T_{j}$ and  $m(x_j \otimes 1) \mapsto x_j$, for  $j=1, \ldots n-i-1$; 
 $m(1\otimes T_j) \mapsto T_{j+n-i}$ and  $m(1\otimes  x_j) \mapsto x_{j+n-i}$, for  $j=1, \ldots i-1$.
 
 Let $\pi$ be a smooth representation of $G_n$. Then $\pi_{\rho_n}$ is an $\calH(M, \rho_{n-i}\boxtimes\rho_i)$-module by restriction from $\calH(G_n,\rho_n)$. Let $\pi_N$ be the normalized Jacquet functor i.e the maximal quotient of $\pi$ such that $N$ acts trivially. Then we have a natural map
  $\pi_{\rho_n} \rightarrow (\pi_N)_{\rho_{n-i}\boxtimes \rho_i}$.

\begin{proposition} \label{prop jacquet} (\cite {BK2} Corollary 7.11)   
 As $\mathcal H(M, \rho_{n-i} \boxtimes \rho_i )$-modules, 
$\pi_{\rho_n}\cong (\pi_N)_{\rho_{n-i}\boxtimes \rho_i}$. 
\end{proposition}

\subsection{Bernstein-Zelevinsky derivatives} \label{ss ZD}
Let $U_{i}$ be the subgroup of $M$ consisting of matrices of the form
\[ \begin{pmatrix}  I_{r(n-i)} & 0 \\ 0 & u \end{pmatrix} ,\]
where $u$ is a strictly upper-triangular matrix in $G_{i}$.  The character $\overline\psi$ of conductor $\mathfrak p$ defines a Whittaker character $\psi$ of $U_{i}$ 
\[ 
\psi (u) =\sum_{j=r(n-i)+1}^{rn-1}  \overline\psi(u_{j,j+1}) 
\] 
where $u_{j,j+1}$ refers to the matrix entries. Let $\omega$ be a smooth $M$-module. Let $\omega_{U_{i}, \psi}$ be
 the space of  $\psi$-twisted $U_{i}$-coinvariants. It is naturally a $G_{n-i}$-module. 
The $ri$-th Bernstein-Zelevinsky derivative of a smooth $G_{n}$-module $\pi$ is defined by 
\begin{align} \label{eqn BZ derivative}
 \pi^{(ri)} = (\pi_N) _{U_{i}, \psi}
\end{align}
Thus the $ri$-th Bernstein-Zelevinsky derivative is a functor from the category of smooth $G_{n}$-modules to the category of smooth $G_{n-i}$-modules. 
We note that  th $l$-th derivative $\pi^{(l)}$ is defined for any non-negative integer $l$, however, if $\pi$ is an object in $\mathfrak{R}^{\mathfrak{s}}(G_{n})$ 
 then $\pi^{(l)}=0$ unless $l$ is divisible by $r$. 

\subsection{Bernstein-Zelevinsky derivative for $\mathcal H_n$} \label{ss bz der hn}

  Abusing notation, we shall identify $\mathcal H_{n-i}$ and $m(\mathcal H_{n-i}\otimes 1)$.  Let $\mathbf S_i\in {\mathcal H}_{i}$ be the sign projector. 
 Let $\mathbf S_i^n=m(1\otimes  {\mathbf S_i})$.  
  Let $\sigma$ be an  $\mathcal H_{n}$-module.  The $i$-th Bernstein-Zelevinsky derivative of $\sigma$ is the natural $\mathcal H_{n-i}$-module 
\[ 
\mathbf{BZ}_i(\sigma): =\mathbf S_i^n(\sigma). 
\]

\begin{theorem} \label{thm bz aha}
Let $\pi$ be an admissible representation of $G_n$ in $\mathfrak{R}^{\mathfrak{s_n}}(G_n)$.  There is a functorial isomorphism 
$\mathbf{BZ}_i(\pi_{\rho_n})\cong (\pi^{(ri)})_{\rho_{n-i}}$ of $\mathcal H_{n-i}$-modules.  

\end{theorem}
\begin{proof} By Proposition \ref{prop jacquet}, $\pi_N$ belongs to the Bernstein component with the type $\rho_{n-i}\boxtimes \rho_i$. It follows that 
$(\pi_N)_{\rho_{n-i}}$ is an admisible representation in $\mathfrak{R}^{\mathfrak{s}_i}(G_i)$. Theorem \ref{thm functorial}, applied to $G_i$, implies that there is an 
isomorphism 
\[ 
\mathbf S_i((\pi_N)_{\rho_{n-i}}) \cong ((\pi_N)_{\rho_{n-i}})_{U_i,\bar\psi_i}.
\] 
Since  the isomorphism is functorial, it is also an isomorphism of $\calH_{n-i}$-modules. Since $((\pi_N)_{\rho_{n-i}})_{U_i,\bar\psi_i}\cong ((\pi_N)_{U_i,\bar\psi_i})_{\rho_{n-i}}\cong (\pi^{(ri)})_{\rho_{n-i}}$, the theorem follows.

\end{proof}

As it is true for Theorem \ref{thm functorial}, in the case of the Bernstein component of representations generated by their Iwahori-fixed vectors, Theorem \ref{thm bz aha} holds without the assumption that $\pi$ is admissible.

 \section{A Leibniz rule} \label{s bz gha}

\subsection{Affine Hecke algebras} \label{def aha}

We shall state the definition of an affine Hecke algebra in a greater generality which will be needed in the following subsections.

Let $(X,R, X^{\vee}, R^{\vee})$ be a root datum where $R$ is a reduced root system and $X$  a $\mathbb{Z}$-lattice containing $R$. 
Let $W$ be the Weyl group of $R$. 
Fix a set of simple roots $\Delta$. The choice of $\Delta$ determines a set $S$ of simple reflections in $W$. Let 
  $l: W \rightarrow \mathbb{Z}$  be the  length function such that $l(s)=1$ for all $s\in S$.  Let $\calA\cong \mathbb C[X]$ be the group algebra of $X$. In other words, 
  $\calA$ has a basis of elements $\theta_x$, $x\in X$, such that $\theta_x \theta_y=\theta_{x+y}$, for all $x,y\in X$. 

\begin{definition} \label{def affine heck alg}
 The affine Hecke algebra $\mathcal H:=\mathcal H(X,  R, \Delta, q)$ associated to the datum is defined to be the complex associative algebra generated by 
 the elements $T_w, w\in W$, and the algebra $\calA$, subject to the relations 
\begin{enumerate}
\item $T_{w}T_{w'}=T_{ww'}$ if $l(ww')=l(w)+l(w')$,
\item $(T_s+1)(T_s-q)=0$ for $s \in S$.
\item $T_s \theta_x - \theta_{s(x)} T_s= (q-1)\frac{\theta_x - \theta_{s(x)}}{1-\theta_{-\alpha^{\vee}}}.$ 
\end{enumerate}
\end{definition}

Denote by $\mathcal H_W$ the finite-dimensional subalgebra of $\mathcal H$ generated by $T_w$ ($w \in W$). 
 We have an isomorphisms of vector spaces $\mathcal H \cong \mathcal A\otimes_{\mathbb C} \mathcal H_W$. 
 Let $\mathbb{T}=\mathrm{Hom}(X, \mathbb{C}^{\times})$. 
The center $\mathcal Z$ of $\mathcal H$ is isomorphic  to $\mathbb C[X]^W$. Hence 
central characters of $\mathcal H$ are parameterized by $W$-orbits in $\mathbb{T}$. We shall denote by $Wt$ the $W$-orbit of $t \in \mathbb{T}$. 
Let $\mathcal J_{Wt}$ be the corresponding maximal ideal in $\mathcal Z$. 
For a finite-dimensional $\mathcal H$-module $\chi$, let $\chi_{[Wt]}$  be the subspace of $\chi$ annihilated by a power of 
 $\mathcal J_{Wt}$. Then 
\[  \chi \cong \bigoplus_{Wt \in \mathbb{T}/W} \chi_{[Wt]} .\]

\smallskip 
Let $X_n=X_n^{\vee} = \bigoplus_{k=1}^n \mathbb{Z}\epsilon_k$ be a $\mathbb{Z}$-lattice. Set $\alpha_{kl}=\epsilon_k-\epsilon_l$ ($k \neq l$) and also set $\alpha_k=\alpha_{k,k+1}$ ($k=1, \ldots, n$). Let $R_n=R_n^{\vee}=\left\{ \epsilon_k-\epsilon_l: l\neq k\right\}$ be a root system of type $A_{n-1}$. 
Let $\Delta_n=\left\{ \epsilon_i-\epsilon_{i+1}: i=1,\ldots, n-1 \right\}$. 
The Iwahori-Hecke algebra $\mathcal H_n$  of $GL(n)$ (from Section \ref{S:hecke}) is isomorphic to $\mathcal{H}(X_n, R_n, \Delta_n, q)$.

\subsection{Lusztig's first reduction theorem} \label{s first l red}

We shall need a variation \cite[Section 2]{OS} of Lusztig's reduction theorem for the affine Hecke algebra $\mathcal H_n$ \cite[Section 8]{Lu}. 
Let $\mathbb{T}_n=\mathrm{Hom}(X_n, \mathbb{C}^{\times})$. Any $t\in \mathbb{T}_n$ is identified with an $n$-tuple $(z_1, \ldots ,z_n)$ of non-zero complex numbers where 
$z_i$ is the value of $t$ at $\epsilon_i$.   Let 
 $\mathbb{T}_r=\mathrm{Hom}(X_n, \mathbb{R}_{> 0})$ and $\mathbb{T}_{un}=\mathrm{Hom}(X_n, S^1)$.   
Any $t \in \mathbb T_n$ has a polar decomposition $t=vu$ where $v \in \mathbb{T}_r$ and $u \in\mathbb{T}_{un}$. 
 Write $x(u)$ for the value of $u$ at $x\in X_n$.  Hence $u=(z_1, \ldots ,z_m)$ where $z_k=\epsilon_k(u)$. 
  We can permute the entries of 
$u$ such that, for a partition $\mathbf n=(n_1, \ldots , n_m)$  of $n$, $z_1=\ldots =z_{n_1} \neq z_{n_1+1} =\ldots $ etc. Let 
\[  
R_{\mathbf n}= \left\{ \alpha \in R_n: \alpha(u)=1  \right\}.  
\] 
It is a root subsystem of $R_n$ which, as the notation indicates, depends on the partition $\mathbf n$. It is isomorphic to the product  $R_{n_1}\times \ldots \times R_{n_m}$. 
Let $S_{\mathbf n}\cong S_{n_1} \times \ldots \times S_{n_m}$ be its Weyl group. 
 Let $\Delta_{\mathbf n}$ be the set of simple roots in $R_{\mathbf n}$ determined by $R^+_{\mathbf n}=R^+_n \cap R_{\mathbf n}$. 
 Let 
 \[ 
 \mathcal H_{\mathbf n}:=\mathcal H(X_n, R_{\mathbf n}, \Delta_{\mathbf n}, q)\cong \mathcal H_{n_1} \otimes \ldots \otimes \mathcal H_{n_m} 
 \] 
  be the associated affine Hecke algebra (Definition \ref{def affine heck alg}). 
  Let $\mathcal Z_{\mathbf n}=\mathcal A_n^{S_{\mathbf n}}$ be the center of $\mathcal H_{\mathbf n}$. Let
  $\mathcal J_{ S_{\mathbf n}t}$ be an ideal in $\mathcal Z_{\mathbf n}$ corresponding to the central character $S_{\mathbf n} t$. 
 Let $\sigma$ be a finite-dimensional ${\mathcal H}_{\mathbf n}$-module annihilated by a power of $\mathcal J_{S_{\mathbf n}t}$. Then 
  $\iota(\sigma)=\mathcal H_n \otimes_{\mathcal H_{\mathbf n}} \sigma$ is annihilated by a power of $\mathcal J_{S_nt}$. 
 
The following result and proof are a variation of \cite[Sections 8.16 and 10.9]{Lu}. 
 \begin{theorem} \label{thm equ cat first red}
 The functor $\iota$ defines an equivalence  between the category of finite-dimensional ${\mathcal H}_{\mathbf n}$-modules annihilated by a power of $\mathcal J_{S_{\mathbf n}t}$ and the category of finite-dimensional $\mathcal H_n$-modules annihilated by a power of $\mathcal J_{S_nt}$. 
 \end{theorem} 

\subsection{First reduction for the Bernstein-Zelevinsky derivatives} \label{s translate BZ}

We keep using notations from the previous subsection. In particular, we fixed $t=vu \in \mathbb{T}_n$, and 
 we have a canonical isomorphism $\mathcal H_{\mathbf n} \cong \mathcal H_{n_1} \otimes \ldots \otimes \mathcal H_{n_m}$, 
 where $\mathbf n=(n_1, \ldots, n_m)$ is a partition of $n$, arising from $u$. 

Fix an integer $i \leq n$. For each $m$-tuple  $\mathbf i=(i_1,\ldots, i_m)$  of integers, such that $i_1+\ldots +i_m=i$ and $0\leq i_k\leq n_k$ ($k=1,\ldots, m$), define
another $m$-tuple $\mathbf n-\mathbf i=(n_1- i_1,\ldots, n_m-i_m)$.  Each pair $(n_k-i_k, i_k)$ gives rise to an embedding 
$\mathcal H_{n_k-i_k} \otimes \mathcal H_{i_k} \subseteq \mathcal H_{n_k}$, as in Section \ref{s:jacquet}, and these combine to give an embedding 
\[ 
\mathcal H_{\mathbf n-\mathbf i} \otimes \mathcal H_{\mathbf i} \subseteq \mathcal H_{\mathbf n}
\] 
where $\mathcal H_{\mathbf i} \cong \mathcal H_{i_1} \otimes \ldots \otimes \mathcal H_{i_m}$ etc. (Note, if $i_k=0$, then the corresponding factor is the trivial 
algebra $\mathbb C$.) Abusing notation, we shall identify $\mathcal H_{\mathbf n-\mathbf i}$ with its image in $\mathcal H_{\mathbf n}$ via the map $h\mapsto h\otimes 1$. 
Let $\mathbf S_{\mathbf i}\in \mathcal H_{\mathbf i}$ be the sign projector in $\mathcal H_{\mathbf i}$, 
and let $\mathbf S_{\mathbf i}^{\mathbf n}$ be the image of $1\otimes \mathbf S_{\mathbf i}$ in $\mathcal H_{\mathbf n}$.  Let $\sigma$ be an $\mathcal H_{\mathbf n}$-module. 
Then $\mathbf S_{\mathbf i}^{\mathbf n}(\sigma)$ is naturally an $\mathcal H_{\mathbf n-\mathbf i}$-module. Thus we have a functor 
\[  \mathbf{BZ}^{\mathbf n}_{\mathbf i}(\sigma): =\mathbf S_{\mathbf i}^{\mathbf n}(\sigma) 
\]
 from the category of $\mathcal H_{\mathbf n}$-modules to the category of ${\mathcal H}_{\mathbf n-\mathbf i}$-modules. 

Observe that $\mathcal H_{\mathbf n-\mathbf i}$ is a Levi subalgebra of $\mathcal H_{n-i}$ and 
$\mathcal H_{\mathbf i}$ is a Levi subalgebra of $\mathcal H_{i}$. We are now ready to state the first reduction result. 

\begin{theorem} \label{thm first red bz}
Let $\pi$ be a finite-dimensional $\mathcal H_n$-module annihilated by a power of $\mathcal J_{S_nt}$. Let $\sigma$ be a finite-dimensional $\mathcal H_{\mathbf n}$-module annihilated by a power of $\mathcal J_{S_{\mathbf n}t}$ such that $\pi \cong \iota(\sigma)$ (see Theorem \ref{thm equ cat first red}). Then there is an isomorphism
\begin{align} \label{eqn bz first red} 
\mathbf{BZ}_i(\pi) \cong \bigoplus_{\mathbf i} \mathcal H_{n-i} \otimes_ {{\mathcal H}_{\mathbf n-\mathbf i}}    \mathbf{BZ}^{\mathbf n}_{\mathbf i} ( \sigma ) 
\end{align}
where the sum is taken over all $m$-tuple of integers $\mathbf i=(i_1,\ldots, i_m)$ satisfying $i_1+\ldots +i_m=i$ and $0\leq i_k\leq n_k$ ($k=1,\ldots, m$).
\end{theorem}

\begin{proof}
Using the Mackey Lemma for affine Hecke algebras (see \cite[Section 2]{Mi} and \cite{Kl}), 
\begin{align}\label{eqn mackey thm} 
\mathrm{res}^{\mathcal H_n}_{\mathcal H_{n-i} \otimes \mathcal H_i} (\mathcal H_{n} \otimes_{\mathcal H_{\mathbf n}} \sigma) \cong 
\bigoplus_{\mathbf i} (\mathcal H_{n- i} \otimes \mathcal H_{ i}) 
\otimes_{(\mathcal H_{\mathbf n-\mathbf i} \otimes \mathcal H_{\mathbf i})} 
 \left(\mathrm{res}^{\mathcal H_{\mathbf n}}_{\mathcal H_{\mathbf n-\mathbf i} \otimes \mathcal H_{\mathbf i}}  \sigma \right)
\end{align}
where the sum is over $\mathbf i$ as in the statement of the theorem. 
We remark that the Mackey Lemma asserts that the composition factors of $ \mathrm{res}^{\mathcal H_n}_{\mathcal H_{n-i} \otimes \mathcal H_i} (\mathcal H_{n} \otimes_{\mathcal H_u} \sigma)$ are precisely those on the right hand side of the above isomorphism. 
The composition factors are indeed direct summands since their $\mathcal H_{n-i} \otimes \mathcal H_i$-central characters are distinct. 
Furthermore, using the Frobenius reciprocity, we have
\begin{align}\label{eqn sign first red} \mathbf S_i^n ((\mathcal H_{n-i} \otimes \mathcal H_{i}) \otimes_{(\mathcal H_{\mathbf n-\mathbf i} \otimes \mathcal H_{\mathbf i})} \sigma )\cong  \mathcal H_{n-i} \otimes_{{\mathcal H}_{\mathbf n-\mathbf i} } \mathbf S_{\mathbf i}^{\mathbf n}(\sigma). 
\end{align}
 Combining (\ref{eqn mackey thm}) and (\ref{eqn sign first red}), we obtain (\ref{eqn bz first red}).
\end{proof}

\begin{remark}
When $\sigma$ is an irreducible $\mathcal H_{\mathbf n}$-module, then $\sigma \cong \sigma_1 \boxtimes \ldots \boxtimes \sigma_m$ for some irreducible $\mathcal H_{n_k}$-modules $\sigma_k$. In this case, 
\[\mathbf{BZ}^{\mathbf n}_{\mathbf i} ( \sigma ) \cong \mathbf{BZ}_{i_1}(\sigma_1) \boxtimes \ldots \boxtimes \mathbf{BZ}_{i_m}(\sigma_m) .
\]
From this viewpoint, Theorem \ref{thm first red bz} can be seen as a Leibniz rule. 
\end{remark}

 \section{Reduction to graded Hecke algebras} \label{s bz ghaII}
 
 \subsection{Graded affine Hecke algebras } 

We shall now need the graded affine Hecke algebra attached to the root datum $(X,R,X^{\vee},R^{\vee})$. 
   Let $V =X\otimes_{\mathbf Z}  \mathbb{C}$.

\begin{definition} \label{def gah}
\cite[Section 4]{Lu}
The graded affine Hecke algebra $\mathbb{H}=\mathbb{H}(V, R, \Delta, \log q)$ is an associative algebra with the unit over $\mathbb{C}$ generated by the symbols $\left\{ t_w :w \in W \right\}$ and $\left\{ f_v: v \in V \right\}$ satisfying the following relations:
\begin{enumerate}
\item[(1)] The map $w  \mapsto t_w$ from $\mathbb{C}[W]=\bigoplus_{w\in W} \mathbb{C}w  \rightarrow \mathbb{H }$ is an algebra injection,
\item[(2)] The map $v \mapsto f_v$ from $S(V) \rightarrow \mathbb{H}$ is an algebra injection, where $S(V)$ is the symmetric algebra of $V$, 
\item[(3)] writing $v$ for $f_v$ from now on, for $\alpha \in \Delta$ and $v \in V$,
\[    vt_{s_{\alpha}}-t_{s_{\alpha}}s_{\alpha}(v)=\log q \cdot \langle v, \alpha^{\vee} \rangle .\]
\end{enumerate}
\end{definition}

In particular, $\mathbb{H} \cong S(V) \otimes \mathbb{C}[W]$ as vector spaces.
We also set $\mathbb{A}=S(V)$, the graded algebra analogue of $\mathcal A$. Let $\mathbb{Z}=\mathbb{A}^{W}$ be the center of $\mathbb{H}$ \cite[Sec. 4]{Lu}. 
Let $V^*=\Hom(X, \mathbb C)$. The central characters of irreducible representations are parameterized by $W$-orbits in $V^*$. If 
$\zeta\in V^*$, let $W\zeta$ denote the corresponding orbit an the central character. Let $\mathbb J_{W\zeta} \subset \mathbb Z$ be the corresponding 
maximal ideal. 

\subsection{Lusztig's second reduction theorem} \label{ss lus srt}

 Let $\mathcal H=\mathcal H( X, R, \Delta,  q)$ be the affine Hecke algebra defined in Section \ref{def aha},
 $\mathcal A\cong \mathbb C[X]$ the commutative subalgebra, and 
  $\mathcal Z\cong \mathbb C[X]^W$ be the center of $\mathcal H$. Let $\mathcal F$ be the quotient field of $\mathcal A$.
   Let $\mathcal H_F \cong \mathcal H_W \otimes_{\mathbb C} \mathcal F$ with the algebra structure naturally extended from $\mathcal H$.

Following Lusztig \cite[Section 5]{Lu}, for $\alpha\in\Delta$, define $\tau_{s_{\alpha}} \in \mathcal H_F$ by 
\[  \tau_{s_{\alpha}}+1=(T_{s_{\alpha}}+1)\mathcal G(\alpha)^{-1} ,\]
where 
\[  \mathcal G(\alpha) = \frac{\theta_{\alpha}q-1}{\theta_{\alpha}-1} \in \mathcal F.
\]
It is shown in \cite[Section 5]{Lu} that the map from $W$ to the units of $\mathcal H_F$ defined by $s_{\alpha} \mapsto \tau_{s_{\alpha}}$ is an injective group homomorphism. 

On the graded Hecke algebra side, let $\mathbb{H}=\mathbb{H}(V, R, \Delta, \log q)$ be as in Definition \ref{def gah}. Let $\mathbb{F}$ be the quotient field of $\mathbb A$ and let $\mathbb{Z}$ be the center of $\mathbb{H}$. Let $\mathbb H_F \cong \mathbb H_W \otimes_{\mathbb C} \mathbb F$ with the algebra structure naturally extended from  $\mathbb H$. 
For $\alpha \in \Delta$, define $\overline{\tau}_{s_{\alpha}} \in \mathbb H_F$ by 
\[  \overline{\tau}_{s_{\alpha}}+1=(t_{s_{\alpha}}+1)g(\alpha)^{-1}, \]
where
\[  g(\alpha)= \frac{\alpha+\log q}{\alpha} \in \mathbb F . 
\]
As in the affine case, the map from $W$ to the units of $\mathbb H_F$ defined by $s_{\alpha} \mapsto \overline{\tau}_{s_{\alpha}}$ is an injective group homomorphism.

 Any  $\zeta \in V^*$ defines $t \in \mathbb{T}=\mathrm{Hom}(X, \mathbb{C}^{\times})$ by 
 $x(t)=e^{x(\zeta)}$, for all $x\in X$.  We shall express this relationship by $t=\exp(\zeta)$. We shall say that  $\zeta$ is {\em real} for the root system $R$ if 
   $\alpha(\zeta) \in \mathbb R$ for all $\alpha\in R$. Then $t=\exp(\zeta)$ satisfies 
  $\alpha(t) >0$, for all $\alpha\in R$. Conversely, every such $t$ arises in this fashion, from a real $\zeta$. 
Let $\widehat{\mathcal Z}$ be the $\mathcal J_{Wt}$-adic completion of $\mathcal Z$ and let $\widehat{\mathbb{Z}}$ be the $\mathbb{J}_{W\zeta}$-adic completion of $\mathbb{Z}$. Let $\widehat{\mathcal H}=\widehat{\mathcal Z} \otimes_{\mathcal Z} \mathcal H$ and let  $\widehat{\mathbb H}=\widehat{\mathbb Z} \otimes_{\mathbb Z} \mathbb H$. Let $\widehat{\mathcal H}_F=\widehat{\mathcal Z} \otimes_{\mathcal Z} \mathcal H_F$ and let  $\widehat{\mathbb H}_F=\widehat{\mathbb Z} \otimes_{\mathbb Z} \mathbb H_F$. Let $\widehat{\mathcal A}=\widehat{\mathcal Z} \otimes_{\mathcal Z} \mathcal A$ and let  $\widehat{\mathbb A}=\widehat{\mathbb Z} \otimes_{\mathbb Z} \mathbb A$. Let $\widehat{\mathcal J}_{Wt}=\widehat{\mathcal Z}\otimes_{\mathcal Z} \mathcal J_{Wt}$ and let $\widehat{\mathbb J}_{W\zeta}=\widehat{\mathbb Z}\otimes_{\mathbb Z}\mathbb J_{W\zeta}$.

\begin{theorem} \label{thm lusztig red thm}\cite[Theorem 9.3, Section 9.6]{Lu} Recall that we are assuming  that $\zeta\in V^*$ is real for the root system $R$. 
\begin{enumerate}
\item There is an isomorphism denoted $j$ between $\widehat{\mathcal H}_F$ and $\widehat{\mathbb{H}}_F$ determined by 
\[  j(\tau_{s_{\alpha}}) = \overline{\tau}_{s_{\alpha}} ,   \quad   j(\theta_{x})=e^{x}. \] 
\item The above map also induces isomorphisms between $\widehat{\mathcal Z}$ and $\widehat{\mathbb Z}$, between $ \widehat{\mathcal A}$ and $\widehat{\mathbb A}$ and between $\widehat{\mathcal H}$ and $\widehat{\mathbb{H}}$. 
\end{enumerate}
\end{theorem}

A crucial point for the proof of (2) is the fact that
\[ \frac{e^{\alpha}q-1}{e^{\alpha}-1} \cdot\frac{\alpha}{\alpha+\log q} \in \mathbb{F}
\]
is holomorphic and non-vanishing at any $\zeta' \in W\zeta$, and hence is an invertible element in $\widehat{\mathbb{A}}$. 

Now (2) gives the following isomorphisms:
\[ \mathcal H/\mathcal J_{Wt}^i\mathcal H \cong \widehat{\mathcal H}/ \widehat{\mathcal J}_{Wt}^i\widehat{\mathcal H} \cong \widehat{\mathbb H}/ \widehat{\mathbb J}_{W\zeta}^i\widehat{\mathbb H} \cong \mathbb H/\mathbb J_{W\zeta}^i\mathbb H 
\]
and hence:

\begin{theorem} \label{thm equiv cat} \cite[Section 10]{Lu} Assume that $\zeta\in V^*$ is real. 
There is an equivalence of categories between the category of finite-dimensional 
$\mathbb{H}$-modules annihilated by a power of $\mathbb J_{W\zeta}$
and the category of finite-dimensional $\mathcal H$-modules annihilated by a power of $\mathcal J_{Wt}$, where $t=\exp(\zeta)$. 
\end{theorem}

Let $\Lambda$ be the functor in Theorem \ref{thm equiv cat}.  Explicitly, for a finite-dimensional $\mathbb H$-module annihilated by a power of $\mathbb J_{W\zeta}$, 
$\Lambda(\pi)$ is equal to $\pi$, as vector spaces,  but the $\mathcal H$-action on $\pi$ is given by 
\[   h \cdot_{\mathcal {H}} x = j (h) \cdot_{\widehat{\mathbb H}} x,
\]
where $h \in \mathcal{H}$ and $x\in \pi$. Note that the functor extends to the category of finite-dimensional $\mathbb H$-modules that are sums of $\mathbb H$-modules, 
where each summand is annihilated by a power of $\mathbb J_{W\zeta}$ for some real $\zeta$. 

 \begin{proposition} \label{prop translate sign} Recall the sign projector $\mathbf S=(\sum_{w \in W} (1/q)^{l(w)})^{-1} \sum_{w \in W} (-1/q)^{l(w)} T_w$ in $\mathcal H$ and let 
$\mathbf s=|W|^{-1}\sum_{w \in W} (-1)^{l(w)} t_w$ be the sign projector in $\mathbb H$.  Then $j(\mathbf S)=a\cdot \mathbf s$, 
where $a$ is an invertible element in $ \widehat{\mathbb A}$.  

\end{proposition}

\begin{proof} Let $\alpha\in \Pi$.  Note that $ \mathbf S (T_{s_{\alpha}} +1) {\mathcal G}(\alpha)^{-1}=0$. Applying $j$ to this equation gives 
\[ 
j(\mathbf S )(t_{s_{\alpha}} +1) g(\alpha)^{-1}=0, 
\] 
hence $j(\mathbf S )(t_{s_{\alpha}} +1)=0$ for $\alpha\in \Pi$. 
This shows that $j(\mathbf S)\in \widehat{\mathbb H}\cdot  {\mathbf s}$. 
 Since $\widehat{\mathbb H}\cdot  {\mathbf s}=\widehat{\mathbb A}\cdot  {\mathbf s}$, we have  $j(\mathbf S)=a\cdot \mathbf s$, for some $a\in \widehat{\mathbb A}$. 
  Using the same argument for $j^{-1}$, we obtain $j^{-1}(\mathbf s)= b \cdot \mathbf S$ for some $b\in \widehat{\mathcal A}$.  Hence $j(b)a=1$  and $a$ is invertible. 
\end{proof}

We have the following corollary to Proposition \ref{prop translate sign}: 

\begin{corollary} \label{cor equality sign }
Let $\pi$ be a finite-dimensional  $\mathbb H$-module annihilated by a power of $\mathcal J_{W\zeta}$, where $\zeta\in V^*$ is real. 
 Identify $\pi$ and $\Lambda(\pi)$ as linear spaces. 
The multiplication by $a\in \widehat{\mathbb A}$ (from Proposition \ref{prop translate sign}) 
 provides a natural isomorphism between the vector spaces  $ \mathbf{s}(\pi)$ and $\mathbf S(\Lambda(\pi))$. 
\end{corollary}

\subsection{Bernstein-Zelevinsky derivatives for graded algebras} 
Let $V_n=X_n \otimes_{\mathbb Z}  \mathbb C$, and $\mathbb{H}_n:=\mathbb{H}(V_n, R_n, \Delta_n, \log q)$. For every $i=0, \ldots, n$, we have a Levi subalgebra 
$\mathbb H_{n-i} \otimes \mathbb H_i$.  Let $\mathbf s_i\in \mathbb H_i$ be the sign projector, and let $\mathbf s^n_i \in \mathbb H_n$ be 
the image of $1\otimes \mathbf s_i$ under the inclusion $\mathbb H_{n-i} \otimes \mathbb H_i\subseteq \mathbb H_n$. 

Let $\pi$ be a finite-dimensional representation of $\mathbb H_n$.  The $i$-th Bernstein-Zelevinsky derivative of $\pi$ is the natural 
$\mathbb H_{n-i}$-module 
\[ 
\mathbf{gBZ}_i(\pi):=\mathbf s_i^n(\pi). 
\] 

Write any $\zeta\in V_n^*= \Hom(X_n, \mathbb C)$ as an $n$-tuple $(\zeta_1, \ldots, \zeta_n)$ where $\zeta_i$ is the value of $\zeta$ on 
the standard basis element $\epsilon_i\in X_n$. In this case $\zeta$ is real for $R_n$ if and only if $\zeta_k-\zeta_l \in \mathbb R$ for all $1\leq k,l\leq n$.  

\begin{theorem} \label{thm real bzder}  Assume that $\zeta\in V_n^*$ is real for the root system $R_n$, and $\pi$ is a finite-dimensional $\mathbb H_n$-module annihilated by a power 
of $\mathbb J_{S_n\zeta}$. There is a natural isomorphism of $\mathcal H_{n-i}$-modules 
$\mathbf{BZ}_i(\Lambda(\pi))$ and $\Lambda(\mathbf{gBZ}_i(\pi))$. 
\end{theorem}
\begin{proof} Note that the functor $\Lambda$ commutes with the restriction to Levi subalgebras, that is, we can either restrict to $\mathbb H_{n-i} \otimes \mathbb H_i$ and 
then apply $\Lambda$, or apply $\Lambda$ and then restrict to $\mathcal H_{n-i} \otimes \mathcal H_i$. Decompose $\pi$ under the action of $\mathbb H_i$ 
\[ 
\pi =\oplus \pi_{[S_i \zeta']} 
\] 
where $\pi_{[S_i \zeta']}$ is the summand annihilated by a power of $\mathbb J_{S_i\zeta'}$. Concretely, the sum runs over $S_i$-orbits of the $i$-tuples $\zeta'$ that appear 
as the tail end of the $n$-tuples in the $S_n$-orbit of $\zeta$. We have the corresponding decomposition for the action of $\mathcal H_i$, 
\[ 
\Lambda(\pi) =\oplus \Lambda(\pi)_{[S_i t']} 
\] 
where $t'=\exp(\zeta')$. (The underlying vector spaces of $\pi_{[S_i \zeta']} $ and $\Lambda(\pi)_{[S_i t']}$ are the same.) It follows that $\Lambda(\pi)_{[S_i t']}$ and $\Lambda( \pi_{[S_i \zeta']} )$ are isomorphic $\mathcal H_{n-i} \otimes \mathcal H_i$-modules. 
Recall that $\mathbf S_i^n= 1\otimes \mathbf S_i$ and $ \mathbf s_i^n=1\otimes \mathbf s_i$, where $\mathbf S_i$ and $\mathbf s_i$ are the sign projectors in 
$\mathcal H_i$ and $\mathbb H_i$, respectively. 
Now we have the following isomorphisms of $\mathcal H_{n-i}$-modules 
\[ 
\mathbf S_i^n(\Lambda(\pi)_{[S_i t']}) \cong \mathbf S_i^n( \Lambda( \pi_{[S_i \zeta']} ))\cong \Lambda( \mathbf s_i^n(\pi_{[S_i\zeta']}))
\] 
where the second is furnished by Corollary \ref{cor equality sign }.  This isomorphism is given by the action of an invertible element in 
$\widehat{\mathbb H}_i$ and therefore intertwines $\mathcal H_{n-i}$-action. 
\end{proof} 

\subsection{Second reduction for Bernstein-Zelevinsky derivatives}

In this section, we transfer the problem of computing Bernstein-Zelevinsky derivatives $\mathbf{BZ}_{\mathbf i}^{\mathbf n}$ in Theorem \ref{thm first red bz} to the corresponding problem for graded Hecke algebras. We retain the notation from Sections \ref{s first l red}  and \ref{s translate BZ}. In particular, $\mathbf n=(n_1, \ldots, n_m)$ is a 
partition of $n$, and  we have fixed $t\in \mathbb{T}_n$ such that  $\alpha(t)>0$ for all $\alpha \in R_{\mathbf n}$. 
Then there exists $\zeta\in V_n^*$, real  for the root system $R_{\mathbf n}$, such that $t=\exp(\zeta)$. 
Let 
\[ \mathbb H_{\mathbf n}:= \mathbb{H}(V_n, R_{\mathbf n}, \Delta_{\mathbf n},  \log q)\cong \mathbb H_{n_1} \otimes \ldots \otimes \mathbb H_{n_m}.
\]
Let $\mathbf i=(i_1, \ldots, i_m)$ be an $m$-tuple of integers such that $0\leq i_k\leq n_k$ for all $k$ and 
 $\mathbf n-\mathbf i=(n_1- i_1,\ldots, n_m-i_m)$.  Each pair $(n_k-i_k, i_k)$ gives rise to an embedding 
$\mathbb H_{n_k-i_k} \otimes \mathbb H_{i_k} \subseteq \mathbb H_{n_k}$, and these combine to give an embedding 
\[ 
\mathbb H_{\mathbf n-\mathbf i} \otimes \mathbb H_{\mathbf i} \subseteq \mathbb H_{\mathbf n}
\] 
where $\mathbb H_{\mathbf i} \cong \mathbb H_{i_1} \otimes \ldots \otimes \mathbb H_{i_m}$ etc. 
Abusing notation, we shall identify $\mathbb H_{\mathbf n-\mathbf i}$ with its image in $\mathbb H_{\mathbf n}$ via the map $h\mapsto h\otimes 1$. 
Let $\mathbf s_{\mathbf i}\in \mathbb H_{\mathbf i}$ be the sign projector in $\mathbb H_{\mathbf i}$, 
and let $\mathbf s_{\mathbf i}^{\mathbf n}$ be the image of $1\otimes \mathbf s_{\mathbf i}$ in $\mathbb H_{\mathbf n}$.  Let $\sigma$ be an $\mathbb H_{\mathbf n}$-module. 
Then $\mathbf s_{\mathbf i}^{\mathbf n}(\sigma)$ is naturally an $\mathbb H_{\mathbf n-\mathbf i}$-module. Thus we have a functor 
\[ 
 \mathbf{gBZ}^{\mathbf n}_{\mathbf i}(\sigma): =\mathbf s_{\mathbf i}^{\mathbf n}(\sigma) 
\]
 from the category of $\mathbb H_{\mathbf n}$-modules to the category of ${\mathbb H}_{\mathbf n-\mathbf i}$-modules. The following is proved in the 
 same way as Theorem \ref{thm real bzder}.

\begin{theorem} \label{thm bz derivative graded case} Let $\zeta\in V^*_n$ be real for the root system $R_{\mathbf n}$. 
 Let $\pi$ be a finite-dimensional $\mathbb{H}_{\mathbf n}$-module annihilated by a power of $\mathbb J_{S_{\mathbf n}\zeta}$.
 Then we have a natural isomorphism of $\mathcal H_{\mathbf n-\mathbf i}$-modules 
\[   \mathbf{BZ}_{\mathbf i}^{\mathbf n}(\Lambda(\pi)) \cong \Lambda(\mathbf{gBZ}_{\mathbf i}^{\mathbf n}(\pi)) . 
\]

\end{theorem}

\section{Bernstein-Zelevinsky derivatives of Speh representations} \label{s bz speh}

\subsection{Speh modules}

Speh representations of $p$-adic groups were studied extensively by Tadi\'c as a part of studying the unitary dual. We recall the definition of 
(generalized) Speh representations. Let $\bar{n}$ be a partition of $n$, write $\bar{n}^t= (e_1, \ldots ,e_f)$, $e_1\geq \ldots \geq e_f$,  where $t$ is the transpose. 
Let $\mathrm{St}_{e_k}$ be the Steinberg representation of $GL(e_k,F)$ and let $\mathrm{St}_{e_k}'=\nu^{\frac{e_k-1}{2}}\mathrm{St}_{e_k}$ be a twist of $\mathrm{St}_{e_k}$, where $\nu(g)=|\mathrm{det}(g)|_F$. Let $P_{\bar{n}}$ be the standard parabolic subgroup associated to the partition $\bar{n}^t$. Let $\rho(g)=|\mathrm{det}(g)|_F^r$ for some complex number $r$. The unique quotient of the induced representation 
\[ 
\pi_{(\bar{n},\rho)}=\mathrm{Ind}_{P_{\bar{n}}}^{GL(n,F)} (\rho\mathrm{St}_{e_1}' \boxtimes \rho\nu^{-1}\mathrm{St}_{e_2}' \cdots \boxtimes \rho\nu^{-f+1}\mathrm{St}_{e_f}')
\] 
 is the generalized Speh representation associated to ($\bar{n}, \rho$).  If $e_1=e_2=\ldots =e_f$ then $\pi_{\bar{n}}$ is a Speh representation. 
 

Under the Borel-Casselman equivalence, generalized Speh representations correspond to $\mathcal H_n$-modules with single $\mathcal H_{S_n}$-type (see \cite{BC}, \cite{BM2}, \cite{CM}). 
Since these $\mathcal H_n$-modules have real infinitesimal character, we can look at the corresponding modules for the graded algebra $\mathbb{H}_n$. 
Following \cite{BC}, we shall construct intrinsically the modules as follows. For $\kappa=-r \log q$, we have the following Jucys-Murphy elements: for $k=2,\ldots, n$,
\begin{align} \label{eqn jucys murphy}
 JM_k:= - p( t_{s_{1,k}}+\dots + t_{s_{k-1,k}}) +\kappa
\end{align}
and $JM_1=\kappa$, where $p=\log q$. 
It is straightforward to check that the maps $\epsilon_k \mapsto JM_k$ and $t_{w} \mapsto t_w$ define an algebra homomorphism from $\mathbb{H}_n$ to $\mathbb{C}[S_n]$. Let $\sigma_{\bar{n}}$ be the irreducible $\mathbb{C}[S_n]$-module corresponding to $\bar{n}$. For example, the partition $(n)$ defines the trivial representation while $(1,\ldots, 1)$ defines the sign representation.
 Let $\sigma_{(\bar{n},\kappa)}$ be the $\mathbb{H}$-module pulled back from $\sigma_{\bar{n}}$ via the map defined above, where $JM_k$ depends on $\kappa$. 
 This is the  {\it generalized Speh module} associated to $(\bar{n},\kappa)$. The module $\sigma_{(\bar{n},\kappa)}$ corresponds to $\pi_{(\bar{n},\rho)}$ under the 
 Borel-Casselman equivalence and the Lusztig equivalence in Theorem \ref{thm equiv cat}.

Recall that $\mathbf{gBZ}_i(\pi)$ is the $i$-th Bernstein-Zelevinsky derivative of an $\mathbb{H}_n$-module $\pi$.

\begin{lemma} \label{lem restrict speh}
Let $\pi$ be the generalized Speh $\mathbb H_n$-module associated to the datum $(\bar{n},\kappa)$. 
Then $\mathbf{gBZ}_{i}(\pi)$ is a direct sum of generalized Speh $\mathbb H_{n-i}$-modules. 
Moreover, $\epsilon_1$ acts by the constant $\kappa$ on each direct summand of $\mathbf{gBZ}_{i}(\pi)$.
\end{lemma}
\begin{proof}
This follows from the construction of generalized Speh modules (see e.g. (\ref{eqn jucys murphy})) and the fact that the category of $\mathbb{C}[S_n]$-modules is semisimple.
\end{proof}

We now recover a result of Lapid-M\'inguez (for the case of generalized Speh modules).

\begin{corollary} \label{cor bz speh}
Let $\pi$ be a generalized Speh representation of $GL(n,F)$ associated to $(\bar{n}, \rho)$. Then $\pi^{(i)}$ is the direct sum of generalized Speh modules associated to 
$(\bar{n}',\rho)$, where $\bar{n}'$ runs through all partitions obtained by removing $i$ boxes from $\bar{n}$ with at most one in each row such that the resulting diagram is still a Young diagram.

\end{corollary}

\begin{proof}
Since $\Lambda(\sigma_{\bar{n},\kappa})=\pi^{I_n}$ 
  it suffices to compute $\mathbf{gBZ}_i(\sigma_{\bar{n},\kappa})$ by Theorem \ref{thm real bzder}. 
From the observation in Lemma \ref{lem restrict speh}, it suffices to determine the $\mathbb{C}[S_{n-i}]$-module structure of $\mathbf{gBZ}_i(\sigma_{\bar{n},\kappa})$, 
 and this follows from a special case of the Littlewood-Richardson rule (or the Pieri's formula). 
\end{proof}

Generalized Speh modules form a subclass of ladder representations defined by Lapid-M\'inguez \cite{LM}. Bernstein-Zelevinsky derivatives of ladder representations are 
computed there using a determinantal formula of Tadi\'c.

\section{Appendix: Projectivity of Gelfand-Graev representation}

In this appendix we shall prove that the Gelfand-Graev representation of a quasi-split reductive group is projective. Roughly speaking, this follows from two facts: 
its Bernstein components are finitely generated, and its dual is injective. 
 
\subsection{Some algebra} 
 Let $\mathcal H$ be a $\mathbb C$-algebra with 1, such that its center $\mathcal Z$ is a noetherian algebra, and $\mathcal H$ is a finitely generated $\mathcal Z$-module. 
 In particular, every finitely generated $\mathcal H$-module $\pi$ is also finitely generated  $\mathcal Z$-module. Hence any ascending chain of submodules of $\pi$ stabilizes. 
 It follows that any finitely generated $\mathcal H$-module has irreducible quotients. 
Assume also that $\mathcal H$ is countably dimensional, as a vector space over $\mathbb C$. Then 
 any finitely generated $\mathcal H$-module $\pi$  is countably dimensional. In this situation the Schur lemma holds, that is, 
 if $\pi$ is irreducible then $\pi$ is annihilated by a maximal ideal $\mathcal J$ in $\mathcal Z$. It follows that any irreducible $\mathcal H$-module is finite dimensional.

 Fix a maximal ideal $\mathcal J$ in $\mathcal Z$. Let $\widehat{\mathcal Z}$ be the $\mathcal J$-adic completion of $\mathcal Z$. It is known that $\widehat{\mathcal Z}$ is a flat 
 ${\mathcal Z}$-module \cite{AMD}. 
 If $\pi$ is a $\mathcal Z$-module, let $\widehat{\pi}$ be the $\mathcal J$-adic completion of $\pi$. If 
 $\pi$ is finitely generated then $\widehat{\pi}\cong \widehat{\mathcal Z}\otimes_{\mathcal Z}\pi$. 
 In particular, the $\mathcal J$-adic completion gives an exact functor from the category of finitely generated $\mathcal H$-modules to the 
 category of finitely generated 
 $\widehat{\mathcal H}\cong  \widehat{\mathcal Z}\otimes_{\mathcal Z}\mathcal H$-modules.

 \begin{theorem} \label{T:projective} 
 Assume that $\mathcal H$ satisfies all conditions spelled out above. 
Let $\pi$ be a finitely-generated $\mathcal H$-module. Suppose 
\[ \mathrm{Ext}^i_{\mathcal H}(\pi, \sigma)=0, ~ i\geq 1, 
\]
for all finite-dimensional $\mathcal H$-modules $\sigma$, or $\mathrm{Hom}_{\mathcal H}(\pi, \cdot )$ is an exact functor on the category of 
finite-dimensional  $\mathcal H$-modules. Then $\pi$ is a projective $\mathcal H$-module. 
\end{theorem}
 \begin{proof} Suppose we have two $\mathcal H$-modules $\sigma, \tau$ with a surjection $f: \sigma \rightarrow \tau$. 
To show $\pi$ is projective, we have to show  that the map $f_*: \mathrm{Hom}_{\mathcal H}(\pi, \sigma) \rightarrow \mathrm{Hom}_{\mathcal H}(\pi, \tau)$ is surjective. 
Note that, since $\pi$ is finitely generated, we can assume, without loss of generality, that $\sigma$ and $ \tau$ are finitely generated.
 We shall prove firstly a local version of the desired result. 

\begin{lemma} \label{lem completion proj} For every  maximal ideal $\mathcal J$ in $\mathcal Z$, the map 
$ \widehat{f}_*: \mathrm{Hom}_{\widehat{\mathcal H}}(\widehat{\pi}, \widehat{\sigma}) \rightarrow  \mathrm{Hom}_{\widehat{\mathcal H}}(\widehat{\pi}, \widehat{\tau})$ is surjective. 
\end{lemma}

\begin{proof}
Let $\psi \in \mathrm{Hom}_{\widehat{\mathcal H}}(\widehat{\pi}, \widehat{\tau})$. Since $\widehat{\tau}$ is naturally isomorphic to the inverse limit of 
$\tau/\mathcal J^i \tau$, the map $\psi$ is given by a system of 
$\psi_i \in \mathrm{Hom}_{\mathcal H}(\pi, \tau/\mathcal J^i \tau)$, commuting with the canonical projections 
$\tau/\mathcal J^{i+1} \tau\rightarrow \tau/\mathcal J^i \tau$. Thus, in order to find 
$\phi\in  \mathrm{Hom}_{\widehat{\mathcal H}}(\widehat{\pi}, \widehat{\sigma})$  such that $ \widehat{f}_*(\phi)=\psi$, it suffices to find a system of 
$\phi_i \in \mathrm{Hom}_{\mathcal H}(\pi, \sigma/\mathcal J^i \sigma)$, commuting with the canonical projections 
$\sigma/\mathcal J^{i+1} \sigma\rightarrow \sigma/\mathcal J^i \sigma$, such that 
$ f_*^i(\phi_i)=\psi_i$, for all $i$, where 
\[ 
f_*^i: \mathrm{Hom}_{\mathcal H}(\pi, \sigma/ \mathcal J^i\sigma)
 \rightarrow \mathrm{Hom}_{\mathcal H}(\pi, \tau/ \mathcal J^i\tau), 
\]
are naturally arising from $f$. Observe that the quotients $\sigma/ \mathcal J^i\sigma$ and $\tau/ \mathcal J^i\tau$ are finite dimensional, 
since $\sigma$ and $\tau$ are finitely generated. Consider the following commutative diagram:
\[\xymatrix{ 
                  0         \ar[r] & \mathrm{Hom}_{{\mathcal H}}({\pi}, \mathcal J^i \sigma/\mathcal J^{i+1}\sigma)               \ar[d]    \ar[r]         &     \mathrm{Hom}_{{\mathcal H}}(\pi, \sigma/\mathcal J^{i+1}\sigma) \ar[d]^{f_*^{i+1}}  \ar[r] & \mathrm{Hom}_{{\mathcal H}}(\pi, \sigma/\mathcal J^i\sigma) \ar[d]^{f_*^i}  \ar[r] & 0                 \\
    0         \ar[r] & \mathrm{Hom}_{\mathcal H}(\pi, \mathcal J^i\tau/\mathcal J^{i+1}\tau )                             \ar[r]                   &         \mathrm{Hom}_{\mathcal H}(\pi, \tau/\mathcal J^{i+1}\tau)  \ar[r] & \mathrm{Hom}_{\mathcal H}(\pi, \tau/\mathcal J^i\tau)   \ar[r] & 0                 \\ }
\]
The assumption on $\pi$ implies that the vertical maps are surjective and the horizontal sequences are exact. We can now construct the sequence $\phi_i$ by induction. Assume 
that $\phi_i$ has been constructed. Let $\phi'_{i+1}$ be any element in $\mathrm{Hom}_{{\mathcal H}}(\pi, \sigma/\mathcal J^{i+1}\sigma)$ in the fiber 
of $\phi_i$. Then $f_*^{i+1}(\phi'_{i+1})$ may not be equal to $\psi_{i+1}$, however, a simple diagram chase shows that  there exists 
$\epsilon \in \mathrm{Hom}_{{\mathcal H}}(\pi, \sigma/\mathcal J^{i+1}\sigma)$, in the image of  
$\mathrm{Hom}_{{\mathcal H}}({\pi}, \mathcal J^i\sigma/\mathcal J^{i+1}\sigma )$, such that 
$\phi_{i+1}=\phi'_{i+1}+\epsilon$  satisfies all requirements.

\end{proof}

We now pass to a global version of the above result.

\begin{lemma}
Let $\pi_1, \pi_2$ be finitely-generated $\mathcal H$-modules. Then $\mathrm{Hom}_{\mathcal H}(\pi_1, \pi_2)$ is a finitely-generated $\mathcal Z$-module.
\end{lemma}

\begin{proof}
Let $x_1,\ldots, x_r$ be a finite set of generators for $\pi$ as $\mathcal Z$-module. Now we consider a $\mathcal Z$-submodule of $\oplus_{k=1}^r \pi$ given by
\[   \left\{ f(x_1),\ldots, f(x_r) : f \in \mathrm{Hom}_{\mathcal H}(\pi_1, \pi_2) \right\}
\]
By the Noetherian property, a submodule of a finitely-generated module is still finitely-generated. This implies the lemma.
\end{proof}

\begin{lemma}
Let $\pi_1, \pi_2$ be finitely-generated $\mathcal Z$-modules, and $ g: \pi_1 \rightarrow \pi_2$ a morphism of $\mathcal Z$-modules. Then the following statements are equivalent:
\begin{enumerate}
\item $g$ is surjective;
\item with respect to every maximal ideal $\mathcal J$ in $\mathcal Z$, the map $\widehat{g}: \widehat{\pi}_1 \rightarrow \widehat{\pi}_2$, arising naturally from $g$, is surjective.
\end{enumerate}
\end{lemma}

\begin{proof}
(1) implies (2) by the exactness of tensoring with $\widehat{\mathcal Z}$. For (2) implying (1), we proceed by contraposition. So assume that (1) fails. Let 
$\sigma$ be an irreducible quotient of $\pi_2/f(\pi_1)$, and let $\mathcal J$ be the annihilator of $\sigma$. Then (2) fails for the $\mathcal J$-adic completion. 
\end{proof}

\begin{lemma} \label{lem hom completion}
Let $\pi_1, \pi_2$ be finitely-generated $\mathcal H$-modules. Then the natural map 
\[ 
\theta: \widehat{\mathcal Z} \otimes_{\mathcal Z} \mathrm{Hom}_{\mathcal H}(\pi_1, \pi_2) \rightarrow 
\mathrm{Hom}_{\widehat{\mathcal H}}(\widehat{\pi_1}, \widehat{ \pi_2}) 
\]
is an isomorphism. 
\end{lemma}

\begin{proof}
Let
\[  \ldots \rightarrow  \mathcal H^{r_2}\rightarrow  \mathcal H^{r_1} \rightarrow \pi_1 \rightarrow 0
\]
be a finitely generated free resolution for $\pi_1$. Then we have the following commutative diagram: 
 \[\xymatrix{ 
                  0         \ar[r] &       \widehat{\mathcal Z}\otimes_{\mathcal Z}\mathrm{Hom}_{\mathcal H}(\pi_1, \pi_2)        \ar[d]^{\theta}   \ar[r]     &    
                   \widehat{\mathcal Z}\otimes_{\mathcal Z}\mathrm{Hom}_{\mathcal H}(\mathcal H^{r_1}, \pi_2)  \ar[d]  \ar[r] &
               \widehat{\mathcal Z}\otimes_{\mathcal Z}\mathrm{Hom}_{\mathcal H}(\mathcal H^{r_2}, \pi_2)     \ar[d]                \\
    0         \ar[r] &    \mathrm{Hom}_{\widehat{\mathcal H}}(\widehat{\pi}_1,\widehat{\pi}_2)                      \ar[r]                   &        
   \mathrm{Hom}_{\widehat{\mathcal H}}(\widehat{\mathcal H}^{r_1}, \widehat{\pi}_2)    \ar[r]   &    \mathrm{Hom}_{\widehat{\mathcal H}}(\widehat{\mathcal H}^{r_2},\widehat{\pi}_2).          \\ }
\]
The first sequence is exact since $\mathrm{Hom}_{\mathcal H}(\cdot, \pi_2)$ is left exact and tensoring with $\widehat{\mathcal Z}$ is exact. The second sequence is exact 
since tensoring with $\widehat{\mathcal Z}$ is exact and $\mathrm{Hom}_{\widehat{\mathcal H}}(\cdot, \widehat{\pi_2})$ is left exact. Since the last two vertical arrows are 
isomorphisms,  $\theta$ is an isomorphism, because kernels of isomorphic maps are isomorphic.
\end{proof}

We can now  prove that $f_*$ is surjective. 
By the previous two lemmas, it suffices to show that $\widehat{f}_*$  is surjective, for every completion. But this is true by Lemma  \ref{lem completion proj}. This completes the 
proof of the theorem. 
\end{proof}

\subsection{Gelfand-Graev representation} Let $G$ be a quasi-split reductive group over a $p$-adic field. 
Let $K$ be a good open compact subgroup of $G$, as in Corollaire 3.9 in  \cite{BD}.  Let $\mathcal H$ be the Hecke algebra of 
compactly supported $K$-biinvaraint functions on $G$. Then, by \cite{BD}, in particular, Corollaire 3.4 there, the algebra $\mathcal H$ satisfies the conditions spelled out at the beginning of this section. Let $\pi$ be a smooth $G$-module. In order to prove that $\pi$ is projective, by Theorem \ref{T:projective}, it suffices to show the following two bullets: 
\begin{itemize} 
\item For every $K$, the summand of $\pi$ generated by $K$-fixed vectors is finitely generated. 
\item The functor $\mathrm{Hom}_G(\pi, \cdot)$ is exact on the category of finite length modules. 
\end{itemize}  
Let $\pi^*$ denote the smooth dual of $G$. Since $\mathrm{Hom}_G(\pi, \sigma^*)\cong \mathrm{Hom}_G(\sigma, \pi^*)$, the second bullet holds if $\pi^*$ is an injective $G$-module. 
This is true if $\pi$ is a Gelfand-Greav representation, by exactness of the Jacquet functor. The first bullet is also true for the Gelfand-Greav representation, by \cite{BH}. Thus 
we have obtained the following corollary: 

\begin{corollary} \label{C:projective} 
 The Gelfand-Graev representation is projective. 
\end{corollary}

 \end{document}